\documentclass[titlepage,12pt]{article} 
\usepackage{hyperref}
\usepackage[usenames,dvipsnames]{pstricks} 
\usepackage{pst-plot} 
\usepackage{amssymb,amsthm,amsmath} 
\usepackage[a4paper]{geometry}
\usepackage{datetime2}
\usepackage[utf8]{inputenc}
\usepackage[italian,english]{babel}
\usepackage{mathrsfs}
\usepackage{bbm}

\date{}

\newcommand{\R}{\mathbb{R}}
\newcommand{\re}{\mathbb{R}}
\newcommand{\Z}{\mathbb{Z}}

\newcommand{\ep}{\varepsilon}

\newcommand{\PMF}{\mathbb{PMF}}
\newcommand{\SPMF}{\mathbb{SPMF}}
\newcommand{\RSPMF}{\mathbb{RSPMF}}
\newcommand{\RSPM}{\mathbb{RSPM}}
\newcommand{\DPMF}{\mathbb{DPMF}}
\newcommand{\RDPMF}{\mathbb{RDPMF}}
\newcommand{\PJ}{P\!J}
\newcommand{\JF}{\mathbb{JF}}

\newcommand{\uep}{u_{\ep}}
\newcommand{\glim}{\Gamma\mbox{--}\!\!\lim}
\newcommand{\loc}{_{\mathrm{loc}}}
\newcommand{\omep}{\omega_\ep}
\newcommand{\logep}{|\log\varepsilon|}
\newcommand{\phihat}{\widehat{\varphi}}
\newcommand{\Ft}{F_\theta}
\newcommand{\St}{\widehat{S}_\theta}
\newcommand{\At}{A_\theta}
\newcommand{\gt}{g_\theta}
\newcommand{\ft}{f_\theta}
\newcommand{\Imm}{\operatorname{Im}}

\newtheorem{thm}{Theorem}[section]

\newtheorem{rmk}[thm]{Remark}
\newtheorem{prop}[thm]{Proposition}
\newtheorem{defn}[thm]{Definition}

\newtheorem{lemma}[thm]{Lemma}
\newtheorem{open}{Open problem}

\title{Symmetry breaking for local minimizers of a free discontinuity problem}

\author{
Massimo Gobbino\vspace{1ex}\\ 
{\normalsize Università di Pisa} \\
{\normalsize Dipartimento di Matematica}\\ 
{\normalsize PISA (Italy)}\\  
{\normalsize e-mail: \texttt{massimo.gobbino@unipi.it}}
\and
Nicola Picenni\vspace{1ex}\\ 
{\normalsize Università di Pisa} \\
{\normalsize Dipartimento di Matematica}\\ 
{\normalsize PISA (Italy)}\\
{\normalsize e-mail: \texttt{nicola.picenni@unipi.it}}
}

\begin{document}

\maketitle

\begin{abstract}

We study a functional defined on the class of piecewise constant functions, combining a jump penalization, which discourages discontinuities, with a fidelity term that penalizes deviations from a given linear function, called the forcing term.

In one dimension, it is not difficult to see that local minimizers form staircases that approximate the forcing term. Here we show that in two dimensions symmetry breaking occurs, leading to the emergence of exotic minimizers whose level sets are not simple stripes with boundaries orthogonal to the gradient of the forcing term.

The proof relies on a suitable adaptation of the calibration method for free discontinuity problems; as a side benefit, our version requires less regularity than the classical one.
\bigskip

\vspace{6ex}

\noindent{\bf Mathematics Subject Classification 2020 (MSC2020):} 
49Q20, 49K05, 49K10

\vspace{6ex}



\end{abstract}


\section{Introduction}

Let $(a,b)\subseteq\re$ be an interval, and let $u:(a,b)\to\re$ be a function. We say that $u$ is a ``pure jump'' function in $(a,b)$, and we write $u\in\PJ((a,b))$, if there exist a real number $c$, a (finite or countable, and possibly also empty) subset $S_u\subseteq(a,b)$, and a function $J:S_u\to\mathbb{R}\setminus\{0\}$ such that
\begin{equation*}
\sum_{x\in S_u} |J(x)| < +\infty,
\end{equation*}
and
\begin{equation}
u(x) = c + \sum_{\substack{y\in S_u \\ y\leq x}} J(y),
\qquad \forall x\in(a,b).
\label{defn:PJ-1d}
\end{equation}

We call $\PJ\loc(\re)$ the set of all functions $u:\re\to\re$ whose restriction to every interval $(a,b)$ belongs to $\PJ((a,b))$. The space $\PJ\loc(\mathbb{R})$ naturally generalizes piecewise constant functions, and can also be characterized as the set of functions in $BV\loc(\re)$ whose distributional derivative is purely atomic. 

It is not difficult to verify that the representation~(\ref{defn:PJ-1d}) is unique for every function $u \in \PJ((a,b))$. Specifically, the constant $c$ equals the limit of $u(x)$ as $x \to a^+$, the set $S_u$ consists of the discontinuity points of $u$ and, for each $x \in S_u$, the function $J(x)$ is the jump of $u$ at $x$, that is, the difference between the right and left limits of $u$.

We refer to the elements of $S_u$ as the \emph{jump points} of $u$, and to $|J(x)|$ as the \emph{jump height} at $x$. This quantity also coincides with the difference between the upper and lower approximate limits of $u$ at $x$, denoted by $u^+(x)$ and $u^-(x)$, respectively.

Given the real parameters
\begin{equation}
\theta\in[0,1), 
\qquad 
\alpha>0, 
\qquad 
\beta>0, 
\qquad 
M\neq 0,
\label{hp:tabM}
\end{equation}
we introduce the jump functional with fidelity term
\begin{equation}
\JF_{\theta,\alpha,\beta,M}(\Omega,u) =
\alpha\sum_{x\in S_u\cap\Omega} |u^+(x)-u^-(x)|^\theta +
\beta\int_\Omega (u(x)-Mx)^2\,dx,
\label{defn:JF-1d}
\end{equation}
defined for every open set $\Omega\subseteq\mathbb{R}$ and every $u\in\PJ\loc(\mathbb{R})$, with values in nonnegative real numbers or even $+\infty$, because the first term might be a diverging series. This functional, extended to $+\infty$ when $u\not\in\PJ\loc(\re)$, is lower semicontinuous with respect to convergence in $L^2(\Omega)$, and more generally in every space $L^p(\Omega)$. 

Minimizing this functional involves a competition between the sum, a sort of regularizing term that penalizes jumps, and the integral, which encourages $u$ to approximate the function $f(x) := Mx$. Consequently, we refer to $f(x)$ as the \emph{forcing term} and to the integral as the \emph{fidelity term}. Notably, in the limiting case $\theta=0$, the sum simply counts the number of jump points of $u$ in $\Omega$, while for $\theta=1$ (which is excluded in (\ref{hp:tabM}) because in that case the functional is not lower semicontinuous) it would represent the total variation of $u$ in $\Omega$.

An \emph{entire local minimizer} of \eqref{defn:JF-1d} is any function $u\in\PJ\loc(\mathbb{R})$ satisfying
\begin{equation*}
\JF_{\theta,\alpha,\beta,M}(\Omega,u) \leq 
\JF_{\theta,\alpha,\beta,M}(\Omega,v)
\end{equation*}
for every open set $\Omega\subseteq\re$, and every $v\in\PJ\loc(\mathbb{R})$ that coincides with $u$ outside a compact subset of $\Omega$.

In one dimension, entire local minimizers can be described rather easily. As we establish in Theorem~\ref{thm:main-1d}, their graphs are ``staircases'' that follow the profile of the forcing term, with steps whose length and height are determined solely by the parameters \eqref{hp:tabM}.

The problem can be generalized to dimensions $d\geq 2$. To this end, we consider the space $\PJ\loc(\mathbb{R}^d)$ of pure jump functions in $\mathbb{R}^d$. Even if this space lacks an elementary representation as \eqref{defn:PJ-1d}, any such function has a jump set $S_u$, which is a $(d-1)$-rectifiable subset of $\mathbb{R}^d$, and a well-defined jump height $|u^+(x)-u^-(x)|$ that is measurable with respect to the $(d-1)$-dimensional Hausdorff measure $\mathcal{H}^{d-1}$ restricted to $S_u$. For the precise functional setting, we refer to Section~\ref{sec:statements} below.

With this notation, the natural generalization of \eqref{defn:JF-1d} is the functional
\begin{equation}
\JF_{\theta,\alpha,\beta,\xi}(\Omega,u) =
\alpha\int_{S_u\cap\Omega} |u^+(x)-u^-(x)|^\theta\,d\mathcal{H}^{d-1} +
\beta\int_\Omega (u(x)-\langle\xi,x\rangle)^2\,dx,
\label{defn:JF-dd}
\end{equation}
defined for every open set $\Omega\subseteq\mathbb{R}^d$ and every $u\in\PJ\loc(\mathbb{R}^d)$. Here, $\theta$, $\alpha$, and $\beta$ are as in \eqref{hp:tabM}, with $\xi\in\mathbb{R}^d\setminus\{0\}$, and $\langle\xi,x\rangle$ denoting the scalar product between $\xi$ and $x$.

The forcing term $f(x) := \langle\xi,x\rangle$ has a one-dimensional profile in the direction of $\xi$, which initially suggested that entire local minimizers might retain a one-dimensional structure in the same direction. However, our main result demonstrates that this is not always (and probably never) the case. Indeed, in Theorem~\ref{thm:main-dd} we show that, in the case $\theta=0$, there exist entire local minimizers that are not one-dimensional. 

This was somewhat surprising, at least to us, but it does not contradict any general principle. In a symmetric problem, symmetry ensures that the symmetric transformation of any solution is still a solution, but it does not necessarily imply that every solution itself must preserve the same symmetries.

\paragraph{\textit{\textmd{Motivation in one dimension}}}

Our interest for this problem originated from our asymptotic analysis of the staircasing phenomenon for the Perona-Malik functional. In order to explain this connection, let us start by considering in one dimension the Perona-Malik functional with fidelity term
\begin{equation}
\PMF(u):=
\int_0^1 \log(1+ u'(x)^2)\,dx +
\beta \int_0^1 (u(x)-f(x))^2\,dx,
\label{defn:PMF}
\end{equation}
where $\beta$ is a positive real number, and $f\in L^2((0,1))$ is a given forcing term. Since the function $p\mapsto\log(1+p^2)$ is not convex and, even more important, its convex hull is identically zero, it is well-known that
$$\inf\left\{\PMF(u):u\in C^{1}([0,1])\right\}=0
\qquad
\forall f\in L^{2}((0,1)).$$

Therefore, in order to obtain more stable models, several regularization of (\ref{defn:PMF}) have been proposed in the last decades. Here we focus on two of them.
\begin{itemize}
    \item The \emph{singular perturbation regularization}, obtained by adding a convex coercive term depending on higher order derivatives. In its simpler version, this leads to the functional
    \begin{equation}
    \SPMF_\ep(u):=
    \int_0^1 \ep^{10}\logep^2 u''(x)^2\,dx+
    \PMF(u),
    \label{defn:SPMF}
    \end{equation}
    defined for every $u\in H^2((0,1))$ and every $\ep\in(0,1)$.

    \item  The \emph{discrete regularization}, obtained by replacing the derivative in (\ref{defn:PMF}) by finite differences. This leads to the functional
    \begin{equation}
    \DPMF_\ep(u):=
    \int_{0}^{1-\ep} \log\left(1+\left(D^\ep u(x)\right)^2\right)\,dx 
    +\beta \int_0^1 (u(x)-f(x))^2\,dx,
    \label{defn:DPMF}
    \end{equation}
    defined for every $\ep\in(0,1)$, where 
    \begin{equation*}
        D^\ep u(x):=\frac{u(x+\ep)-u(x)}{\ep}
        \qquad
        \forall x\in(0,1-\ep)
    \end{equation*}
    is the classical finite difference, and the domain of the functional is now restricted to the functions $u$ that are piecewise constant with respect to the $\ep$-grid, namely such that
    \begin{equation*}
        u(x)=u(\ep\lfloor x/\ep\rfloor)
        \qquad
        \forall x\in[0,1].
    \end{equation*}

\end{itemize}

Both choices lead to well-posed models, in the sense that for every admissible value of $\ep$ the corresponding minimum problem admits at least one solution. On the other hand, the unstable character of (\ref{defn:PMF}) comes back in the limit as $\ep\to 0^+$, so that minimum values tend to~0, minimizers tend to $f$ in $L^2(\Omega)$ and, more important, minimizers develop a microstructure known as \emph{staircasing effect}. A quantitative analysis of this effect was carried on by the authors in~\cite{FastPM-CdV,FastPM-CdV2} for the singular perturbation, and by the second author in~\cite{fastpm-discreto} for the discrete approximation.

In both cases the main idea consists in zooming-in the graph of minimizers within a window of a suitable size $\omep$. More precisely, given a family of minimizers $\{\uep\}$, and a family $x_\ep\to x_0\in(0,1)$, one considers the family of blow-ups
$$v_\ep(y):=\frac{\uep(x_\ep+\omep y)-f(x_\ep)}{\omep}
\qquad\quad
\forall y\in I_\ep:=\left(-\frac{x_\ep}{\omep},\frac{1-x_\ep}{\omep}\right).$$

The choice of $\omep$ depends on the model. In the case of the singular perturbation, the correct choice is $\omep:=\ep|\log\ep|^{1/2}$, and with a change of variable in the integrals one can see that $v_\ep$ minimizes the \emph{rescaled singular perturbation} of the Perona-Malik functional with fidelity term, defined as
\begin{equation*}
    \RSPMF_\ep(I_\ep,v):=\RSPM_\ep(I_\ep,v)+
    \beta \int_{I_\ep}(v(y)-f_\ep(y))^2\,dy,
\end{equation*}
where
\begin{equation*}
    \RSPM_\ep(\Omega,v):=\int_{\Omega} \left\{\ep^6 v''(y)^2+ 
    \frac{1}{\omep^2} \log\left(1+v'(y)^2\right)\right\} dy, 
\end{equation*}
and the new forcing term is
\begin{equation}
    f_\ep(y):=\frac{f(x_\ep+\omep y)-f(x_\ep)}{\omep}
    \qquad
    \forall y\in I_\ep.
    \label{defn:fep}
\end{equation}

Now, let us assume that $f$ is of class $C^1$. Then, by passing to the limit in (\ref{defn:fep}), we obtain that $f_\ep(y)$ tends to the linear function $y\mapsto f'(x_0)y$ uniformly on bounded sets. Under this assumption, one can establish two key results (see~\cite[Theorem~3.2 and Proposition~4.6]{FastPM-CdV}).
\begin{itemize}

    \item (Gamma convergence). There exists a positive constant $\alpha$  such that, for every bounded open set $\Omega\subseteq\re$,
    \begin{equation}
        \glim_{\ep\to 0^+}\RSPMF_\ep(\Omega,v)=
        \JF_{1/2,\alpha,\beta,f'(x_0)}(\Omega,v).
        \label{Gconv:RSPMF}
    \end{equation}

    \item  (Compactness). The family $\{v_\ep\}$ is relatively compact in $L^2(\Omega)$ for every bounded open set $\Omega\subseteq\re$.

\end{itemize}

After these two key facts have been established, one can conclude in a rather standard way that every limit point of ${v_\ep}$ is an entire local minimizer of the limit functional, which naturally leads to the problem of classifying such minimizers.

In the case of the discrete approximation the situation is analogous. The correct choice is $\omep=(\ep |\log\ep|)^{1/3}$, in which case $v_\ep$ minimizes the \emph{rescaled discrete} Perona-Malik functional with fidelity term, defined as
\begin{equation*}
    \RDPMF_\ep(I_\ep,v) :=
    \frac{1}{\omep^2}\int_{I_\ep'} \log\left(1+ D^{\ep/\omep} v(x)^2 \right) dx +
    \beta \int_{I_\ep} (v(x)-f_\ep(x))^2\,dx,
\end{equation*}
where, in analogy with the previous case, we have set
\begin{equation*}
    I_\ep:=\left(-\frac{x_\ep}{\omep},\frac{1-x_\ep}{\omep}\right)
    \qquad\text{and}\qquad
    I_\ep':=\left(-\frac{x_\ep}{\omep},\frac{1-\ep-x_\ep}{\omep}\right).
\end{equation*}

Again the family $\{v_\ep\}$ is relatively compact in $L^2(\Omega)$ for every bounded open set $\Omega\subseteq\re$, while now (\ref{Gconv:RSPMF}) becomes
\begin{equation*}
    \glim_{\ep\to 0^+}\RDPMF_\ep(\Omega,v)=
    \JF_{0,\alpha,\beta,f'(x_0)}(\Omega,v).
\end{equation*}

As a consequence, again any limit point of $\{v_\ep\}$ is an entire local minimizer to a functional such as (\ref{defn:JF-1d}), just with exponent $\theta=0$ instead of $\theta=1/2$.

\paragraph{\textit{\textmd{Motivation in higher dimension}}}

The previous theory for the Perona-Malik functional can be extended to higher dimension. The Perona-Malik functional can be defined in analogy with (\ref{defn:PMF}), just by replacing the interval $(0,1)$ with a product of intervals or a suitable bounded open set, and $u'(x)$ with the norm of the gradient of $u$. The singular perturbation approximation can be defined in analogy with (\ref{defn:SPMF}), just by replacing $|u''(x)|$ with some norm of the Hessian matrix of $u$. The discrete approximation can be defined in analogy with (\ref{defn:DPMF}) by exploiting some discrete version of the gradient.

We never wrote down the details explicitly, but at least in the case of the singular perturbation, both the Gamma-convergence (see~\cite{2014-M3AS-BelChaGol,PicNic:PhD,Solci}) and the compactness results should still hold, although the proof involves additional technical difficulties. This would suffice to show that, in any space dimension, the limits of blow-ups of minimizers are again entire local minimizers of the functional~\eqref{defn:JF-dd}, with $\theta = 1/2$ and $\xi = \nabla f(x_0)$, where $x_0$ is the limit of the centers of the zoom-in windows. This, in turn, motivates the classification of such entire local minimizers. However, the appearance of the exotic candidates presented in this paper (as well as others whose existence we suspect) significantly complicates this crucial step, even in two dimensions.

\paragraph{\textit{\textmd{Overview of the technique}}}

The characterization of entire local minimizers in one dimension (Theorem~\ref{thm:main-1d}) is essentially an extension of \cite[Proposition~4.5]{FastPM-CdV} to more general exponents, and can be established through fairly elementary arguments, as was done in that earlier work.

In higher dimensions, we rely on two distinct tools. The first is the \emph{slicing technique} (see Proposition~\ref{prop:slicing}), which is effective when we start with an entire local minimizer in $\re^{d_1}$ and wish to extend it to $\re^{d_1 + d_2}$ by simply ignoring the additional variables. This method applies both when proving that staircases are entire local minimizers in any space dimension, and when extending our exotic minimizers from dimension $d = 2$ to dimensions $d \geq 3$.

The second tool is the \emph{calibration method}, originally introduced in the context of free discontinuity problems by G.~Alberti, G.~Bouchitté, and G.~Dal Maso in~\cite{2003-CalcVar-ABDM}. In a nutshell, the idea is to define a new functional $\mathbb{G}(\Omega, u)$ as the flux of a vector field $\Phi$ across the boundary of the hypograph of $u$ in $\Omega$. The key point lies in choosing the vector field $\Phi$ so that the following three conditions are met.
\begin{itemize}

\item \textit{Divergence-free.} The field $\Phi$ must be divergence free. This ensures that $\mathbb{G}(\Omega, v) = \mathbb{G}(\Omega, w)$ whenever $v$ and $w$ coincide in a neighborhood of $\partial\Omega$.

\item \textit{Lower bound}. For every $v \in \PJ_{\mathrm{loc}}(\re^d)$, one requires $\mathbb{G}(\Omega, v) \leq \JF(\Omega, v)$ (here for the sake of shortness we do not write all parameters as in (\ref{defn:JF-dd})). This typically leads to a set of inequalities that the components of $\Phi$ must satisfy.

\item \textit{Matching on the candidate}. For the candidate $u$ to be an entire local minimizer, one requires $\mathbb{G}(\Omega, u) = \JF(\Omega, u)$. This typically results in equalities that must be satisfied by the components of $\Phi$.

\end{itemize}

These three conditions together imply that
\begin{equation}
    \JF(\Omega,v)\geq \mathbb{G}(\Omega,v)=\mathbb{G}(\Omega,u)=\JF(\Omega,u) \label{osculating}
\end{equation}
for every function $v$ that coincides with $u$ in a neighborhood of $\partial\Omega$, which is enough to prove that $u$ is actually an entire local minimizer.

In this paper, we apply the calibration method in two distinct contexts. The first is to provide an alternative proof that staircases are entire local minimizers in one dimension. In this case, we need a divergence-free vector field in $\re^2$, and any such field can be written as the rotated gradient of a scalar function $F$. Thus, in this model case, the entire construction reduces to finding a scalar function $F$ of two variables that satisfies a suitable system of equalities and inequalities (see Proposition~\ref{prop:calibr-1d}).

The second use of the calibration method occurs in verifying that certain exotic ``double staircases'' are indeed entire local minimizers in $\re^2$ for $\theta = 0$, as needed in the proof of Theorem~\ref{thm:main-dd}. Here, we need a divergence-free vector field in $\re^3$, and it is well known that such a field can be expressed as the curl of another vector field of the form $(A, B, 0)$. The one-dimensional construction guides our choice of components: specifically, we set $B(x, y, z) = F(x, z)$, where $F$ is the same function used in the calibration of one-dimensional staircases. This reduces the problem to selecting the function $A$ (see Proposition~\ref{prop:calibr-dd}).

We emphasize that our approach differs from that in~\cite{2003-CalcVar-ABDM} in a key aspect that could be interesting in itself: rather than focusing directly on the vector field $\Phi$, we work instead with the underlying functions, namely $F$ in one dimension, and $A$ and $B$ in two dimensions. In particular, all the conditions we impose take the form of equalities and inequalities involving these functions themselves, not their derivatives. This leads to significantly weaker regularity requirements. For instance, in the one-dimensional case, we do not even require $F$ to be continuous (see the proof of Proposition~\ref{prop:calibr-1d} and the following Remark~\ref{rmk:ABDM-1d}), whereas in~\cite{2003-CalcVar-ABDM} the components of $\Phi$, which in our setting correspond to derivatives of~$F$, are required to be bounded and approximately continuous. Similarly, in two dimensions, we do not impose any differentiability on $A$, while the approach of~\cite{2003-CalcVar-ABDM} would require its derivatives to satisfy analogous regularity conditions (see Proposition~\ref{prop:calibr-dd} and Remark~\ref{rmk:ABDM-2d}).

\paragraph{\textit{\textmd{Other examples of symmetry breaking}}}

Determining whether the solutions of some problem inherits the same symmetries of the problem itself is a very classical problem in analysis. Several famous examples of symmetry breaking, together with some techniques to prove symmetry, are illustrated in the expository paper~\cite{Kawohl-symmetry_or_not}. Among the classical examples, we mention Newton's body of minimal resistance (see~\cite{Symm-Newton}), the non-symmetric groundstates of~\cite{nonsymm-gs}, and the symmetry breaking for the minimizers of some Poincaré-Wirtinger type inequalities (see~\cite{DGS-Wirtinger,GGR-COCV-2018}).

Further classical examples of symmetry breaking arise from the Steiner problem, which is also an example in which minimality can be proved via calibration (see, for example, \cite{PP-BLMS-2022} and the references therein).

Some more recent results in contexts similar to ours are presented in the series of papers \cite{GR-CalcVar,DR-ARMA,DR-SIAM}, where the authors consider functionals defined on sets involving a competition between a perimeter-type attractive term and a repulsive non-local term. In \cite{DR-CalcVar,DKR-JFA} similar models for diffuse interfaces instead of sets were also considered. In both cases, it turns out that minimizers display one-dimensional patterns (periodic stripes), thus exhibiting less symmetries than the energy.

Finally, we mention the problem of symmetry for optimizers in the Caffarelli-Kohn-Nirenberg inequality \cite{CKN}. After some cases of symmetry breaking were discovered in \cite{CKN1}, the problem of determining the exact range of parameters for which this phenomenon occurs remained open for a while, until it was finally settled in \cite{CKN2}. More recently, the same question has been investigated also for the fractional version of this inequality (see \cite{frac_CKN1,frac_CKN2}), but at present only partial results are available in this direction.

\paragraph{\textit{\textmd{Structure of the paper}}}

This paper is organized as follows. In Section~\ref{sec:statements}, we fix the notation and state our main results. In Section~\ref{sec:1d}, we prove the characterization of entire local minimizers in one dimension, introducing in particular our version of the calibration method in this setting. In Section~\ref{sec:staicases-dd}, we recall the classical slicing technique and apply it to show that staircases remain minimizers in all space dimensions, and to reduce the search for exotic minimizers to dimension two. Section~\ref{sec:exotic} forms the core of the paper: here, we construct asymmetric entire local minimizers in two dimensions and establish their properties via a more delicate calibration, whose construction is nonetheless inspired by the one dimensional case. Finally, in Section~\ref{sec:open}, we present some open problems.


\setcounter{equation}{0}
\section{Notation and statements}\label{sec:statements}

\paragraph{\textmd{\textit{Pure jump functions, jump sets and jump heights}}}

Let $d$ be a positive integer, and let $\Omega\subseteq\re^d$ be an open set. Throughout this paper, we consider the usual space $SBV(\Omega)$ of special bounded variation functions, and the space $GSBV(\Omega)$ of all measurable functions $u:\Omega\to\re$ whose truncations
\begin{equation*}
    u_T(x):=\min\{\max\{u(x),-T\},T\}
    \qquad
    \forall x\in\Omega
\end{equation*}
belong to $SBV(\Omega)$ for every $T>0$. At this point, one can introduce the space
\begin{equation*}
    \PJ(\Omega):=\left\{u\in GSBV(\Omega):\nabla u(x) =0 \mbox{ for almost every }x\in\Omega\strut\right\}
\end{equation*}
of \emph{pure jump} functions in $\Omega$, and finally the space $\PJ\loc(\re^d)$ consisting of all measurable functions $u:\re^d\to\re$ whose restriction to every \emph{bounded} open set $\Omega\subseteq\re^d$ belongs to $\PJ(\Omega)$. For the theory of these function spaces we refer to~\cite{AFP}.

In the sequel we need the fact that, for every function \( u \in \PJ\loc(\mathbb{R}^d) \), the approximate limsup and liminf, denoted by \( u^+(x) \) and \( u^-(x) \), respectively, coincide and are finite for every \( x \in \mathbb{R}^d \) except on a set \( S_u \), called the \emph{jump set} of \( u \).  The jump set \( S_u \) is \((d-1)\)-rectifiable, and the \emph{jump height} \( u^+(x) - u^-(x) \) is measurable with respect to the restriction of the \((d-1)\)-dimensional Hausdorff measure \( \mathcal{H}^{d-1} \) to \( S_u \). As a consequence, the functional~(\ref{defn:JF-dd}) is well-defined (possibly taking the value \( +\infty \)) for every \( u \in \PJ\loc(\mathbb{R}^d) \) and for every admissible choice of the parameters.

Incidentally, we recall that for \( \mathcal{H}^{d-1} \)-almost every \( x \in S_u \), the values \( u^+(x) \) and \( u^-(x) \) also coincide with the approximate limits of \( u \) taken from the two sides of \( S_u \).

\paragraph{\textmd{\textit{Staircases}}}

Let us recall the notation introduced in~\cite{FastPM-CdV} in order to describe staircase-like functions (see \cite[Definitions~2.3 and~2.4]{FastPM-CdV}).

\begin{defn}[Staircases]\label{defn:staircase}
\begin{em}

Let $S:\re\to\re$ be the function defined by
\begin{equation}
S(x):=2\left\lfloor\frac{x+1}{2}\right\rfloor
\qquad
\forall x\in\re,
\label{defn:S(x)}
\end{equation}
where, for every real number $\alpha$, the symbol $\lfloor\alpha\rfloor$ denotes the greatest integer less than or equal to $\alpha$. 
\begin{itemize}

\item For every pair $(H,V)$ of real numbers, with $H>0$, we call \emph{canonical $(H,V)$-staircase} the function $S_{H,V}:\re\to\re$ defined by
\begin{equation}
S_{H,V}(x):=V\cdot S(x/H)
\qquad
\forall x\in\re.
\nonumber
\end{equation}

\item We say that $v$ is an \emph{oblique translation} of $S_{H,V}$ if there exists a real number $\tau_{0}\in[-1,1]$ such that
\begin{equation}
v(x)=S_{H,V}(x-H\tau_{0})+V\tau_{0}
\qquad
\forall x\in\re.
\nonumber
\end{equation}

\item  In dimension $d\geq 2$, the $(H,V)$-canonical staircase in a direction $\xi\in\re^d$, with $\|\xi\|=1$, is the function
\begin{equation*}
    S_{H,V,\xi}(x):=V\cdot S(\langle x,\xi\rangle/H)
    \qquad
    \forall x\in\re^d,
\end{equation*}
and its oblique translations are of the form
\begin{equation*}
    v(x)=S_{H,V,\xi}(x-H\tau_{0}\xi)+V\tau_{0}
    \qquad  
    \forall x\in\re^d.
\end{equation*}

\end{itemize}

\end{em}
\end{defn}

Roughly speaking, the graph of $S_{H,V}$ is a staircase with steps of horizontal length $2H$ and vertical height $2V$. The origin is the midpoint of the horizontal part of one of the steps. The staircase degenerates to the null function when $V=0$, independently of the value of~$H$. Oblique translations correspond to moving the origin along the line $Hy=Vx$. For a pictorial description of these staircases, we refer to Figure~\ref{figure:staircase}, which is taken from~\cite[Figure~2]{FastPM-CdV}.

\begin{figure}[h]
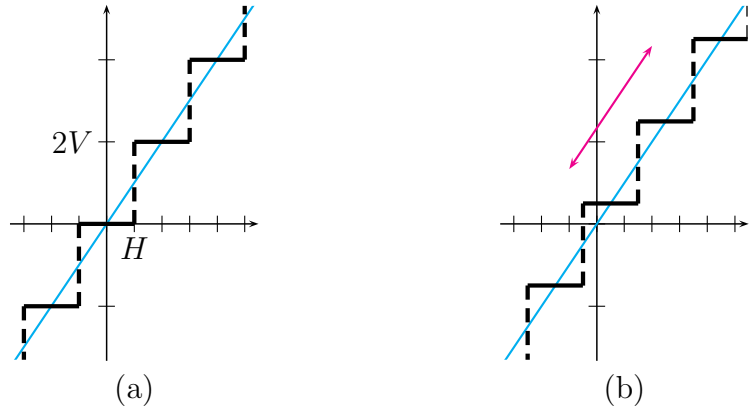


\hfill
\psset{unit=2ex}
\pspicture(-4,-6.5)(6,8.5)

\psline[linewidth=0.7\pslinewidth]{->}(-3.5,0)(5.5,0)
\psline[linewidth=0.7\pslinewidth]{->}(0,-5)(0,8)

\multiput(0,-3)(0,3){4}{\psline[linewidth=0.5\pslinewidth](-0.3,0)(0.3,0)}
\multiput(-3,0)(1,0){9}{\psline[linewidth=0.5\pslinewidth](0,-0.3)(0,0.3)}

\psclip{\psframe[linestyle=none](-3.5,-5)(5.5,8)}
\psline[linecolor=cyan](-4,-6)(6,9)
\multiput(-3,-3)(2,3){4}{\psline[linewidth=2\pslinewidth](0,0)(2,0)}
\multiput(-3,-3)(2,3){5}{\psline[linewidth=2\pslinewidth,linestyle=dashed](0,0)(0,-3)}
\endpsclip

\rput[r](-0.5,3){$2V$}
\rput[t](1,-0.5){$H$}
\rput[t](1,-5.5){(a)}

\endpspicture
\hfill
\pspicture(-4,-6.5)(6,8.5)

\psline[linewidth=0.7\pslinewidth]{->}(-3.5,0)(5.5,0)
\psline[linewidth=0.7\pslinewidth]{->}(0,-5)(0,8)

\multiput(0,-3)(0,3){4}{\psline[linewidth=0.5\pslinewidth](-0.3,0)(0.3,0)}
\multiput(-3,0)(1,0){9}{\psline[linewidth=0.5\pslinewidth](0,-0.3)(0,0.3)}

\psline[linecolor=magenta]{<->}(-1,2)(2,6.5)

\psclip{\psframe[linestyle=none](-3.55,-5)(5.5,8)}
\psline[linecolor=cyan](-4,-6)(6,9)
\multiput(-2.5,-2.25)(2,3){4}{\psline[linewidth=2\pslinewidth](0,0)(2,0)}
\multiput(-2.5,-2.25)(2,3){5}{\psline[linewidth=2\pslinewidth,linestyle=dashed](0,0)(0,-3)}
\endpsclip

\rput[t](1,-5.5){(b)}

\endpspicture
\hfill
\mbox{}

\caption{(a)~Canonical staircase. (b)~Oblique translation with parameter $\tau_{0}=1/2$.}
\label{figure:staircase}
\end{figure}

\paragraph{\textmd{\textit{Main results}}}

The first result of this paper is the complete characterization of entire local minimizers for the functional (\ref{defn:JF-1d}).

\begin{thm}[Entire local minimizers in one dimension]\label{thm:main-1d}

    Let $\theta$, $\alpha$, $\beta$, $M$ be as in (\ref{hp:tabM}). Let us set
    \begin{equation}
        H:=\left(\frac{3(1-\theta)\alpha}{(2|M|)^{2-\theta}\beta}\right)^{1/(3-\theta)}
        \qquad\quad\text{and}\quad\qquad
        V:=MH.
        \label{defn:HV}
    \end{equation}

    Then the set of entire local minimizers of the functional (\ref{defn:JF-1d}) coincides with the set of all oblique translations of the canonical $(H,V)$-staircase. 
\end{thm}

The following remarks clarify some aspects of this result.

\begin{rmk}[Some heuristics]\label{rmk:heuristic}
    \begin{em}
        Let us consider the canonical $(H,V)$-staircase. Each jump has height $2V$, and hence its contribution to the functional (\ref{defn:JF-1d}) is $\alpha(2V)^\theta$. Since the distance between two consecutive jumps is $2H$, we can say that the contribution of jumps, or equivalently of vertical parts of the steps, per unit length is $\alpha(2V)^\theta/(2H)$.

        The contribution of each horizontal step to the fidelity term in (\ref{defn:JF-1d}) is
        \begin{equation*}
            \beta\int_{-H}^H(Mx)^2\,dx=
            \frac{2}{3}\beta M^2 H^3,
        \end{equation*}
        and hence the contribution of the horizontal parts of the steps per unit length is $\beta M^2H^2/3$. If we assume that $V=MH$, which is reasonable if we want a staircase with the same average slope as the forcing term, the sum of the two unitary contributions is
        \begin{equation*}
            \alpha(2H)^{\theta-1}|M|^\theta+\frac{\beta M^2H^2}{3}.
        \end{equation*}

        The value of $H$ that minimizes this expression is exactly the one given in (\ref{defn:HV}).
    \end{em}
\end{rmk}

\begin{rmk}[The limit case $\theta=1$]
    \begin{em}
        The value of $H$ tends to 0 as $\theta \to 1^-$. This aligns with the intuition that, as $\theta$ approaches 1, it becomes increasingly convenient for minimizers to distribute their variation over a greater number of jumps, thereby better adapting to the forcing term. As a further evidence one could prove that, in the limit case $\theta=1$, the unique entire local minimizer is the forcing term $Mx$ itself.
    \end{em}
\end{rmk}

Now we consider the higher dimensional case. The following result, and in particular statement~(2), is the main contribution of this paper.

\begin{thm}[Entire local minimizers in higher dimensions]\label{thm:main-dd}

    Let $d \geq 2$ be an integer, and let $\theta$, $\alpha$, and $\beta$ be as in~\eqref{hp:tabM}. Let $\xi \in \mathbb{R}^d \setminus \{0\}$, set $M := \|\xi\|$, and define $H$ and $V$ as in~\eqref{defn:HV}.

    Then the following statements hold.
    \begin{enumerate}
    \renewcommand{\labelenumi}{(\arabic{enumi})}
        \item All oblique translations of the $(H,V)$-staircase in the direction $\xi/M$ are entire local minimizers of the functional (\ref{defn:JF-dd}).

        \item  If $\theta=0$, then there do exist entire local minimizers of the functional (\ref{defn:JF-dd}) that are not oblique translations of the $(H,V)$-staircase in the direction $\xi/M$.
    \end{enumerate}

\end{thm}

\begin{rmk}[The trivial forcing term]\label{rmk:M=0}
\begin{em}

In Theorem~\ref{thm:main-1d} we assumed that \( M \neq 0 \), and in Theorem~\ref{thm:main-dd} that \( \xi \neq 0 \); both conditions ensure that the forcing term in (\ref{defn:JF-1d}) and (\ref{defn:JF-dd}) is nontrivial. 

When the forcing term is identically zero, it is straightforward to show that the constant function zero is an entire local minimizer. The question is whether it is the unique one. In one dimension, the answer is affirmative and can be proved as in~\cite[Proposition~4.5]{FastPM-CdV}, although the argument given there is specific to the one-dimensional setting. In higher dimensions, the situation is less clear, and uniqueness does not seem to follow from the techniques developed in this paper.

\end{em}    
\end{rmk}


\setcounter{equation}{0}
\section{The one-dimensional case (proof of Theorem~\ref{thm:main-1d})}\label{sec:1d}

The plan of the proof is the following.

\begin{itemize}
    \item In the first step, we exploit a homothety argument in order to decrease the number of parameters. This allows to reduce ourselves to the case where
    \begin{equation}
        \alpha=\alpha_\theta:=\frac{2^{2-\theta}}{1-\theta},
        \qquad\qquad
        \beta=3,
        \qquad\qquad
        M=1,
        \label{defn:abM-simple}
    \end{equation}
    for which (\ref{defn:HV}) yields $H=V=1$, and therefore the candidates to be entire local minimizers are the basic staircase $S(x)$ defined in (\ref{defn:S(x)}) and its oblique translations.

    \item In the second step, we introduce the calibration method in one dimension. This reduces the problem of verifying that $S(x)$ is an entire local minimizer to finding a function of two variables satisfying a suitable set of equalities and inequalities.

    \item In the third step, we explicitly construct the calibration and verify that it meets the required conditions. This completes the first part of the proof, namely, the fact that all oblique translations of $S(x)$ are entire local minimizers.

    \item In the fourth and final step, we prove the converse: any entire local minimizer must be an oblique translation of $S(x)$.

\end{itemize}

\paragraph{\textit{\textmd{Step 1 -- Reduction of the parameters}}}

Up to replacing $u$ by $-u$, we can always assume that $M>0$. Now let $A$ be a positive real number, and for every $u\in\PJ\loc(\re)$ let us set
\begin{equation*}
    u_A(x):=\frac{A}{M}\cdot u\left(\frac{x}{A}\right)
    \qquad
    \forall x\in\re.
\end{equation*}

One can check that $u_A\in\PJ\loc(\re)$, and $u_A$ has a jump point in $x$ with jump height $J$ if and only if $u$ has a jump point in $Ax$ with jump height $MJ/A$. Combining this remark with a change of variable in the integral of the fidelity term, we deduce that
\begin{equation*}
    \JF_{\theta,\alpha,\beta,M}((-L,L),u)=
    \frac{\beta M^2}{3A^3}\cdot
    \JF_{\theta,\widehat{\alpha},3,1}((-AL,AL),u_A),
\end{equation*}
where
\begin{equation*}
    \widehat{\alpha}:=\frac{3\alpha}{\beta}\cdot\frac{A^{3-\theta}}{M^{2-\theta}}.
\end{equation*}

As a consequence, $u$ is an entire local minimizer for the functional (\ref{defn:PJ-1d}) with parameters $(\theta,\alpha,\beta,M)$ if and only if $u_A$ is an entire local minimizer for the same functional with parameters $(\theta,\widehat{\alpha},3,1)$. In particular, if we choose $A:=1/H$, with $H$ given by (\ref{defn:HV}), we have reduced the problem to showing that the set of entire local minima for the functional (\ref{defn:JF-1d}), with parameters given by (\ref{defn:abM-simple}), coincides with the set of oblique translations of the basic staircase $S(x)$.

\paragraph{\textit{\textmd{Step 2 -- The calibration method in one dimension}}}

The key tool is the following.

\begin{prop}[Calibration in one dimension]\label{prop:calibr-1d}
    Let us assume that there exists a function $\Ft:\re^2\to\re$ that satisfies the following two equalities
    \begin{gather}
        \Ft(2j+1,2j)-\Ft(2j-1,2j)=2
        \qquad
        \forall j\in\Z,
        \label{hp:FO=}
        \\
        \Ft(2j-1,2j)-\Ft(2j-1,2j-2)=\frac{4}{1-\theta}
        \qquad
        \forall j\in\Z,
        \label{hp:FV=}
    \end{gather}
    and the following two inequalities
    \begin{gather}
        \Ft(x_2,z)-\Ft(x_1,z)\leq (x_2-z)^3-(x_1-z)^3
        \qquad
        \forall x_1\leq x_2,
        \quad
        \forall z\in\re,
        \label{hp:FO<}
        \\
        \Ft(x,z_2)-\Ft(x,z_1)\leq
        \frac{2^{2-\theta}}{1-\theta}|z_2-z_1|^\theta
        \qquad
        \forall x\in\re,
        \quad
        \forall (z_1, z_2)\in\R^2.
        \label{hp:FV<}
    \end{gather}

    Then the staircase $S(x)$ of Definition~\ref{defn:S(x)}, together with all its oblique translations, is an entire local minimizer for the functional (\ref{defn:JF-1d}), with parameters given by (\ref{defn:abM-simple}).
\end{prop}

\begin{proof}

To begin with, we observe that $S(x)$ minimizes in some interval $(a,b)$ with respect to its boundary conditions if and only if $S(x-\tau_0)+\tau_0$ minimizes in $(a+\tau_0,b+\tau_0)$. As a consequence, if $S(x)$ is an entire local minimizer, then all its oblique translations are automatically entire local minimizers as well. Therefore, in the sequel we focus on $S(x)$.

For every positive integer $k$, we set $a_k:=-(2k+1)$ and $b_k:=2k+1$. We observe that $a_k$ and $b_k$ are jump points of the staircase $S(x)$, and that $S(x)=-2k$ in a right neighborhood of $a_k$ and $S(x)=2k$ in a left neighborhood of $b_k$. For the sake of shortness, we simply write $\JF(\Omega,u)$ to denote the functional (\ref{defn:JF-1d}) with parameters given by (\ref{defn:abM-simple}).

We claim that
\begin{eqnarray}
    \JF((a_k,b_k),v)\geq
    \Ft(b_k,2k)-\Ft(a_k,-2k)= 
    \JF((a_k,b_k),S)
    \label{th:prop-1d}
\end{eqnarray}
for every function $v\in PJ((a_k,b_k))$ that coincides with $S$ outside a compact subset of $(a_k,b_k)$. Since the intervals of the form $(a_k,b_k)$ exhaust the whole real line, this is enough to prove that $S$ is an entire local minimizer.

To begin with, we consider the case in which the jump set of $v$ is finite, and consists of the points $x_1<\ldots<x_m$ for some positive integer $m$. We set $x_0:=a_k$ and $x_{m+1}:=b_k$, and for every $i\in \{0,1,\dots,m\}$ we call $v_i$ the value of $v$ in the interval $(x_i,x_{i+1})$.

With these notations we obtain that
\begin{eqnarray*}
\JF((a_k,b_k),v) & = & 
\frac{2^{2-\theta}}{1-\theta} \sum_{i=0} ^{m-1} |v_{i+1}-v_{i}|^\theta 
+\sum_{i=0} ^{m} 3\int_{x_i} ^{x_{i+1}} (v_i-x)^2\,dx
\\[0.5ex]
&=& 
\frac{2^{2-\theta}}{1-\theta} \sum_{i=0}^{m-1} |v_{i+1}-v_{i}|^\theta + \sum_{i=0} ^{m}\left[(x_{i+1}-v_i)^3 - (x_i-v_i)^3\right].
\end{eqnarray*}

Now we estimate from below the terms of the first sum by exploiting inequality(\ref{hp:FV<}) with $(x,z_1,z_2):=(x_{i+1},v_i,v_{i+1})$, and the terms of the second sum by exploiting inequality (\ref{hp:FO<}) with $(x_i,x_{i+1},v_i)$ instead of $(x_1,x_2,z)$. We deduce that
\begin{eqnarray*}
\JF((a_k,b_k),v) & \geq & 
\sum_{i=0} ^{m-1} \left[\Ft (x_{i+1},v_{i+1})-\Ft(x_{i+1},v_{i})\right]
+ \sum_{i=0} ^{m} \left[\Ft(x_{i+1},v_i)- \Ft(x_{i},v_i)\right]
\\
& = & \Ft(x_{m+1},v_m)-\Ft(x_0,v_0)
\\
& = & \Ft(b_k,2k)-\Ft(a_k,-2k),
\end{eqnarray*}
where the last equality follows from the fact that $v(x)$ coincides with $S(x)$ near the endpoints $a_k$ and $b_k$. This proves the inequality in (\ref{th:prop-1d}).

On the other hand, exploiting a similar telescopic structure, we can write
\begin{eqnarray*}
\Ft(b_k,2k)-\Ft(a_k,-2k) & = & \Ft(2k+1,2k)-\Ft(-2k-1,-2k)   
\\
& = & 
\sum_{j=-k}^{k}\left[\Ft(2j+1,2j)-\Ft(2j-1,2j)\right]
\\
&   & 
\mbox{}+\sum_{j=-k+1}^{k} \left[\Ft(2j-1, 2j)-\Ft(2j-1,2j-2)\right].
\end{eqnarray*}

From (\ref{hp:FO=}) we deduce that 
\begin{equation*}
    \Ft(2j+1,2j)-\Ft(2j-1,2j)=
    2=
    3\int_{2j-1}^{2j+1}(2j-x)^2\,dx,
\end{equation*}
and hence the first sum is equal to the fidelity term of $\JF((a_k,b_k),S)$. From (\ref{hp:FV=}) we deduce that the second sum has $2k$ terms, all of which are equal to $4/(1-\theta)$, and hence the second sum is equal to
\begin{equation*}
    2k\cdot\frac{4}{1-\theta}=
    2k\cdot\frac{2^{2-\theta}}{1-\theta}\cdot 2^\theta,
\end{equation*}
which is exactly the contribution of jump points to $\JF((a_k,b_k),S)$. This proves the equality in (\ref{th:prop-1d}).

Finally, the general case in which $v$ has infinitely many jump points follows by a standard approximation argument, because functions with a finite number of jump points are dense in energy. More precisely, any $v\in PJ((a_k,b_k))$ which coincides with the staircase $S$ in a neighborhood of the endpoints can be approximated with a sequence $\{v_n\}\subseteq PJ((a_k,b_k))$ of functions with finitely many jump points and coinciding with $S$ in a neighborhood of the endpoints, in such a way that
$$\lim_{n\to +\infty} \JF((a_k,b_k),v_n) =
\JF((a_k,b_k),v).$$

This is enough to conclude that (\ref{th:prop-1d}) holds also in the general case.
\end{proof}

\begin{rmk}[Geometric interpretation]\label{rmk:calibration-1d}
\begin{em}

The telescopic argument in the proof can be resumed informally as follows.

The idea is to regard the staircase \( S \) and its competitor \( v \) as unions of horizontal and vertical segments in the plane. The central term in (\ref{th:prop-1d}) is equal to the difference between the values of \( \Ft \) at the points \( (a_k, -2k) \) and \( (b_k, 2k) \), which are the common endpoints of both \( S \) and \( v \). This difference can, in turn, be decomposed as the sum of the differences computed in each horizontal and vertical segment along the paths defined by \( S \) and \( v \).

For the staircase function \( S \), the equalities (\ref{hp:FO=}) and (\ref{hp:FV=}) imply that each of these contributions exactly matches the corresponding term in \( \JF((a_k, b_k), S) \). For the competitor \( v \), the inequalities (\ref{hp:FO<}) and (\ref{hp:FV<}) show that each segment contributes less than or equal to its counterpart in \( \JF((a_k, b_k), v) \). This justifies the inequality in (\ref{th:prop-1d}).

\end{em}    
\end{rmk}

\begin{rmk}[Comparison with~\cite{2003-CalcVar-ABDM}]\label{rmk:ABDM-1d}
\begin{em}

We point out that in Proposition~\ref{prop:calibr-1d} we did not assume any regularity of \( \Ft \). If \( \Ft \) is sufficiently regular, we can define a differential form \( \omega \) and a vector field \( \Phi \) by
\begin{equation*}
    \omega := \frac{\partial \Ft}{\partial x}(x,z)\,dx +
              \frac{\partial \Ft}{\partial z}(x,z)\,dz
    \qquad\text{and}\qquad
    \Phi(x,z) := \left(\frac{\partial \Ft}{\partial z}(x,z),
                        -\frac{\partial \Ft}{\partial x}(x,z)\right).
\end{equation*}

In this case, the differences in values of \( \Ft \) can be interpreted as line integrals of \( \omega \) along the segments, or equivalently as the flux of \( \Phi \) across the same segments (with suitable orientations). This observation connects our approach with the one in~\cite{2003-CalcVar-ABDM}.

\end{em}    
\end{rmk}

\paragraph{\textit{\textmd{Step 3 -- Construction of the calibration}}}

Let us consider the cubic
\begin{equation}
    \varphi_\theta(\sigma):=(3-\theta)\sigma-(1-\theta)\sigma^3.
    \label{defn:cubic-true}
\end{equation}

An elementary calculation shows that it is an increasing function in the interval between its two stationary points $\pm\sigma_\theta$, where
\begin{equation*}
    \sigma_\theta:=\sqrt{\frac{3-\theta}{3(1-\theta)}}.
\end{equation*}

At this point we can introduce the \emph{truncated cubic} $\phihat_\theta:\re\to\re$ defined as
\begin{equation}
    \phihat_\theta(\sigma):=
    \begin{cases}
        \varphi_\theta(-\sigma_\theta)\quad &  
        \text{if }\sigma\leq -\sigma_\theta,
        \\
        \varphi_\theta(\sigma) &
        \text{if }\sigma\in [-\sigma_\theta,\sigma_\theta],
        \\
        \varphi_\theta(\sigma_\theta) &  
        \text{if }\sigma\geq \sigma_\theta,
        \label{defn:cubic}
    \end{cases}
\end{equation}
and finally the function
\begin{equation}
    \Ft(x,z):=
    \dfrac{1}{1-\theta}\left[(3-\theta)x+\phihat_\theta(z-x)\right]
    \qquad
    \forall (x,z)\in\re^2.
    \label{defn:calibration-1d}
\end{equation}

We observe that $\Ft$ is piecewise $C^\infty$, and of class $C^1$ on the whole $\re^2$, because the truncation of the cubic was performed at its stationary points.

We claim that $\Ft$ satisfies (\ref{hp:FO=}) through (\ref{hp:FV<}). The verification of the equalities (\ref{hp:FO=}) and (\ref{hp:FV=}) is immediate from (\ref{defn:calibration-1d}), (\ref{defn:cubic}), and the fact that $\sigma_\theta\geq 1$. As for (\ref{hp:FO<}), we observe that for every $z\in\re$ the partial derivative of (\ref{defn:calibration-1d}) with respect to $x$ is given by
\begin{equation*}\label{eq:dF/dx}
    \frac{\partial\Ft}{\partial x}(x,z)=
    \begin{cases}
        3(z-x)^2 &
        \text{if } z-\sigma_\theta\leq x\leq z+\sigma_\theta,
        \\[1ex]
        \dfrac{3-\theta}{1-\theta} &
        \text{if either } x\leq z-\sigma_\theta \text{ or } x\geq z+\sigma_\theta,
    \end{cases}
\end{equation*}
and in particular
\begin{equation}\label{est:dF/dx<}
    \frac{\partial\Ft}{\partial x}(x,z)
   =\min\left\{3(z-x)^2,\frac{3-\theta}{1-\theta}\right\}
    \leq 3(z-x)^2
    \qquad
    \forall (x,z)\in\re^2,
\end{equation}
so that inequality (\ref{hp:FO<}) follows by integrating over $[x_1,x_2]$.

It remains to prove (\ref{hp:FV<}), which, by setting $z_1=x+a$ and $z_2=x+b$, reduces to
\begin{equation}
    \phihat_\theta(b)-\phihat_\theta(a)\leq
    2^{2-\theta}|b-a|^\theta
    \qquad
    \forall (a,b)\in\re^2.
    \label{hp:phi-holder}
\end{equation}

Due to the monotonicity of $\phihat_\theta$, in the proof of this inequality we can assume that $a<b$. In addition, we can assume also that $-\sigma_\theta\leq a<b\leq \sigma_\theta$, because we can always replace $a$ by $\max\{a,-\sigma_\theta\}$, and $b$ by $\min\{b,\sigma_\theta\}$, and in this way we reduce the right-hand side without altering the left-hand side. After this reduction we are left to proving that
\begin{equation*}
    (3-\theta)(b-a)-(1-\theta)(b^3-a^3)\leq 2^{2-\theta}(b-a)^\theta
    \qquad
    \forall a<b.
\end{equation*}

To this end, we start from the standard inequalities
\begin{equation*}
    e^x\geq 1+x
    \qquad
    \forall x\in\re
\end{equation*}
and 
\begin{equation*}
    \log x=\frac{1}{2}\log(x^2)\leq\frac{1}{2}(x^2-1)
    \qquad
    \forall x>0,
\end{equation*}
and for every $B>0$ we obtain that
\begin{multline*}
    \qquad
    B^\theta=B\cdot B^{\theta-1}=
    B\cdot e^{(\theta-1)\log B}\geq
    B\cdot\left(1-(1-\theta)\log B\right)
    \\[0.5ex]
    \geq
    B\left[1+(1-\theta)\frac{1-B^2}{2}\right]=
    \frac{B}{2}\left(3-\theta-(1-\theta)B^2\right).
    \qquad
\end{multline*}

From this inequality, applied with $B:=(b-a)/2$, we deduce that
\begin{equation*}
    2^{2-\theta}(b-a)^\theta=
    4\left(\frac{b-a}{2}\right)^\theta\geq
    (3-\theta)(b-a)-\frac{1}{4}(1-\theta)(b-a)^3,
\end{equation*}
and we conclude by observing that
\begin{equation*}
    (b-a)^3\leq 4(b^3-a^3)
    \qquad
    \forall a\leq b,
\end{equation*}
because the latter is equivalent to
\begin{equation*}
    4(b^3-a^3)-(b-a)^3=3(b-a)(b+a)^2\geq 0,
\end{equation*}
which is clearly true whenever $b\geq a$.

\begin{rmk}[Motivation Behind the Calibration]
\begin{em}

Let us comment on the choice of the calibration.

\begin{itemize}
    \item In Proposition~\ref{prop:calibr-1d}, the equalities (\ref{hp:FO=}) and (\ref{hp:FV=}) are assumed only for integer values of $j$, but the function $\Ft$ defined in (\ref{defn:calibration-1d}) actually satisfies both equalities for every real value of $j$. As a consequence, the function $\Ft$ defined in (\ref{defn:calibration-1d}) does not calibrate only the canonical staircase, but simultaneously calibrates all its oblique translations.

    In addition, the inequality (\ref{hp:FO<}) is actually an equality whenever $z-\sigma_\theta\leq x_1\leq x_2\leq z+\sigma_\theta$.

    \item  It can be proved (though it is a nontrivial exercise) that every function $\Ft$ satisfying the assumptions of Proposition~\ref{prop:calibr-1d} must coincide with the one defined in (\ref{defn:calibration-1d}) for every $(x,z)$ in the strip $x-1 \leq z \leq x+1$, namely the strip covered by the oblique translations of $S(x)$.

    Outside this strip, alternative constructions should be possible by modifying the choice of the truncated cubic in~(\ref{defn:cubic}). What is actually needed for the proof is that the truncated function coincides with the cubic $\varphi_\theta$ on the interval $[-1,1]$, satisfies a Hölder-type inequality as in~(\ref{hp:phi-holder}), and has a derivative that is bounded in such a way that~(\ref{est:dF/dx<}) holds.

    \item  As often happens in the calculus of variations, the calibration $\Ft$ can be interpreted as a \emph{value function} in the following sense. If the canonical staircase $S(x)$ is an entire local minimizer, then on every interval $(0,\ell)$ it minimizes the functional (\ref{defn:JF-1d}) subject to its own boundary conditions. This suggests defining $\Ft$ on the pairs $(\ell, S(\ell))$ as the corresponding minimum value, which is simply the value of (\ref{defn:JF-1d}) with $\Omega = (0,\ell)$ and $u = S$. If we interpolate these values in the simplest linear way, we recover (\ref{defn:calibration-1d}), at least within the strip covered by all oblique translations of $S$.
\end{itemize}
    
\end{em}    
\end{rmk}

\paragraph{\textit{\textmd{Step 4 -- Uniqueness}}}

In order to prove that the oblique translations of $S(x)$ are the unique entire local minimizers, we need to repeat, for a generic exponent $\theta\in[0,1)$, the same procedure used in~\cite[Section~6.2]{FastPM-CdV} for the case $\theta=1/2$, and in~\cite[Proposition~4.4]{fastpm-discreto} for the case $\theta=0$. Since the full argument is detailed in those references, here we limit ourselves to sketching the main points.
\begin{itemize}

    \item \emph{Discreteness of jump points.} The set of jump points of any entire local minimizer is discrete. Indeed, if this were not the case, one could construct a better competitor by concentrating all sufficiently small jump heights into a single jump point. A key role in this argument is played by the subadditivity of the function $\sigma\mapsto\sigma^\theta$. 
    
    \item \emph{Existence of jump points.} Any interval of sufficiently large length contains at least one jump point, since any entire local minimizer must follow the profile of the forcing term \( Mx \).

    \item \emph{Symmetry of jumps.} At each jump point, any entire local minimizer transitions between two values whose mean equals the forcing term. This necessary condition for minimality corresponds to the Euler--Lagrange equation associated with \emph{horizontal perturbations}, where the location of a jump point is varied.

    \item \emph{Equidistance of jump points.} The distance between any two consecutive jump points is constant. This condition arises from considering \emph{vertical perturbations} of the form \( u + tv \).

    \item \emph{Optimization of the parameters.} After the previous steps, one knows that any entire local minimizer has a staircase structure that follows the forcing term. What remains is to optimize the length (and hence the height) of the steps. This leads to a calculation similar to that in Remark~\ref{rmk:heuristic}.
\end{itemize}


\setcounter{equation}{0}
\section{The slicing method}\label{sec:staicases-dd}

In this section, we show that every entire local minimizer in some dimension can be extended to any higher dimension by simply ignoring the extra variables. We use this result in two instances. First, it implies that statement~(1) of Theorem~\ref{thm:main-dd} follows directly from Theorem~\ref{thm:main-1d}. Second, it reduces the proof of statement~(2) of Theorem~\ref{thm:main-dd} to the case \( d = 2 \).

The idea is the following. Let $d_1$ and $d_2$ be two positive integers. Let us write the elements of $\re^{d_1+d_2}$ as pairs $(x,y)$ with $x\in\re^{d_1}$ and $y\in\re^{d_2}$. Every vector $\xi\in\re^{d_1}$ can be extended to a vector $\widehat{\xi}\in\re^{d_1+d_2}$ by setting $\widehat{\xi}:=(\xi,0)$. Every function $u:\re^{d_1}\to\re$ can be extended to a function $\widehat{u}:\re^{d_1+d_2}\to\re$ by setting $\widehat{u}(x,y):=u(x)$.

\begin{prop}[Extension of entire local minimizers to higher dimension]\label{prop:slicing}

Let $d_1$ and $d_2$ be two positive integers. Let us assume that $u\in\PJ\loc(\re^{d_1})$ is an entire local minimizer for the functional (\ref{defn:JF-dd}) for some choice of the parameters $\theta$, $\alpha$, $\beta$ and $\xi\in\re^{d_1}$.

Then $\widehat{u}\in\PJ\loc(\re^{d_1+d_2})$ is an entire local minimizer for the functional (\ref{defn:JF-dd}) with parameters $\theta$, $\alpha$, $\beta$ and $\widehat{\xi}\in\re^{d_1+d_2}$.
    
\end{prop}

\begin{proof}

Let us consider any open set $\Omega\subseteq\re^{d_1+d_2}$, and any function $v\in\PJ(\Omega)$ that coincides with $\widehat{u}$ outside a compact subset of $\Omega$.

For every $y\in\re^{d_2}$ we can consider the corresponding $d_1$-dimensional section $\Omega_y\subseteq\re^{d_1}$ of $\Omega$, defined as
\begin{equation*}
    \Omega_y:=\left\{x\in\re^{d_1}:(x,y)\in\Omega\right\},
\end{equation*}
and the $d_1$-dimensional sections of $\widehat{u}$ and $v$ defined as
\begin{equation*}
    \widehat{u}_y(x):=\widehat{u}(x,y)=u(x)
    \qquad
    \forall x\in\re^{d_1},
\end{equation*}
and
\begin{equation*}
    v_y(x):=v(x,y)
    \qquad
    \forall x\in\Omega_y.
\end{equation*}

Since $v_y$ coincides with $u$ outside a compact subset of $\Omega_y$, and $u$ is an entire local minimizer in $\re^{d_1}$, we deduce that (here we write $\JF$ without all the parameters, since they are fixed throughout the proof)
\begin{equation*}
    \JF(\Omega_y,v_y)\geq\JF(\Omega_y,u)
    \qquad
    \forall y\in\re^{d_2},
\end{equation*}
and hence
\begin{equation}
    \int_{\re^{d_2}}\JF(\Omega_y,v_y)\,dy\geq
    \int_{\re^{d_2}}\JF(\Omega_y,u)\,dy.
    \label{ineq:slicing}
\end{equation}

The key point is that, on the one hand, we have
\[
\int_{\mathbb{R}^{d_2}} \JF(\Omega_y, u) \, dy = \JF(\Omega, \widehat{u}),
\]
since \( S_{\widehat{u}} = S_u \times \mathbb{R}^{d_2} \) and \( \widehat{u}^{\pm}(x, y) = u^{\pm}(x) \) for every \( (x, y) \in S_u \times \mathbb{R}^{d_2} \). On the other hand, from \cite[Theorem~3.2.22]{Federer-book}, we deduce that
\[
\int_{\mathbb{R}^{d_2}} \JF(\Omega_y, v_y) \, dy \leq \JF(\Omega, v).
\]

Plugging these two relations into (\ref{ineq:slicing}) we conclude that
\begin{equation*}
    \JF(\Omega,\widehat{u})\leq\JF(\Omega,v).
\end{equation*}

Since $\Omega$ and $v$ are arbitrary, this is enough to prove that $\widehat{u}$ is an entire local minimizer in $\re^{d_1+d_2}$.    
\end{proof}


\setcounter{equation}{0}
\section{Exotic minimizers in the plane}\label{sec:exotic}

In this section we prove statement~(2) of Theorem~\ref{thm:main-dd} by showing that for $\theta=0$ some exotic ``double staircases'' are entire local minimizers in $\re^2$. The construction can then be extended to any dimension $d\geq 3$ by a straightforward application of the slicing technique of Proposition~\ref{prop:slicing}.

\paragraph{\textit{\textmd{Step 1 -- Reduction of the parameters}}}

To begin with, up to a rotation we can always assume that $\xi$ is of the form $(M,0)$ for some positive real number $M$. Then, as in Step~1 of the proof of Theorem~\ref{thm:main-1d}, with a homothety argument we reduce ourselves to the case where
\begin{equation}
    \alpha=\alpha_\theta=\frac{2^{2-\theta}}{1-\theta},
    \qquad\qquad
    \beta=3,
    \qquad\qquad
    \xi=(1,0),
\label{defn:abxi-simple}
\end{equation}
for which in Theorem~\ref{thm:main-dd} we obtain $H=V=1$. Therefore, in this case we already know that the canonical staircase in the direction $(1,0)$, as well as all its oblique translations, are entire local minimizers. Our goal is showing that there are more. 

\paragraph{\textit{\textmd{Step 2 -- Definition of the bi-staircase}}}

Let us consider the function $\gt:[0,1]\to\re$ defined as
\begin{equation}
    \gt(x):=\frac{2}{1-\theta}+3x-3x^2
    \qquad
    \forall x\in[0,1].
    \label{defn:gt}
\end{equation}

We observe that
\begin{equation*}
    \frac{2}{1-\theta}\leq
    \gt(x)\leq
    \frac{2}{1-\theta}+\frac{3}{4}<
    \alpha_\theta
    \qquad
    \forall x\in[0,1],
\end{equation*}
because the last inequality is equivalent to $2^{4-\theta}+3\theta-11>0$, which is true for every $\theta\in[0,1)$ because the left-hand side is a convex function of $\theta$ that vanishes for $\theta=1$ with negative derivative.

As a consequence, we can consider the function $\ft:\re\to\re$ defined as
\begin{equation}
    \ft(x):=\int_0^{|x|}\frac{\gt(x)}{\sqrt{\alpha^2 _\theta-\gt(x)^2}}\,dx
    \qquad
    \forall x\in[-1,1],
    \label{defn:ft}
\end{equation}
and then extended to the whole real line by 2-periodicity.

\begin{defn}[Bi-staircases]\label{defn:bi-staircase}
\begin{em}

For every $\theta\in[0,1)$, let $\gt$ and $\ft$ be the functions defined in (\ref{defn:gt}) and (\ref{defn:ft}), and let $S(x)$ be the staircase defined in (\ref{defn:S(x)}). The \emph{canonical bi-staircase} with parameter $\theta$ is the function $\St\in\PJ\loc(\re^2)$ defined by
\begin{equation*}
    \St(x,y):=
    \begin{cases}
    S(x) &
    \text{if }y>\ft(x),
    \\
    S(x-1)+1 \quad &
    \text{if }y<\ft(x).
    \end{cases}
\end{equation*}
 
The \emph{oblique translations} of the canonical bi-staircase are all functions $v$ of the form $v(x,y):=\St(x-\tau_0,y)+\tau_0$ for some real number $\tau_0\in[-1,1]$. The translations of $v(x,y)$ in the $y$-direction are the functions of the form $(x,y)\mapsto v(x,y+t)$ for some $t\in\re$.

\end{em}    
\end{defn}

We observe that the range of the canonical bi-staircase is the set of all integers. Specifically, for every $k\in\Z$ it turns out that $\St(x,y)=2k$ if $x\in(2k-1,2k+1)$ and $y>\ft(x)$, while $\St(x,y)=2k+1$ if $x\in(2k,2k+2)$ and $y<\ft(x)$. The jump set of $\St$ is the union of the graph of $\ft$, and of the vertical half-lines
\begin{equation*}
    \{2k\}\times(-\infty,0]
    \qquad\quad\text{and}\quad\qquad
    \{2k+1\}\times[\ft(1),+\infty)
\end{equation*}
with $k\in\Z$. Figure~\ref{fig:bi-staircase} provides a description of the level sets and the jump set of the canonical bi-staircase.

\begin{figure}[ht]
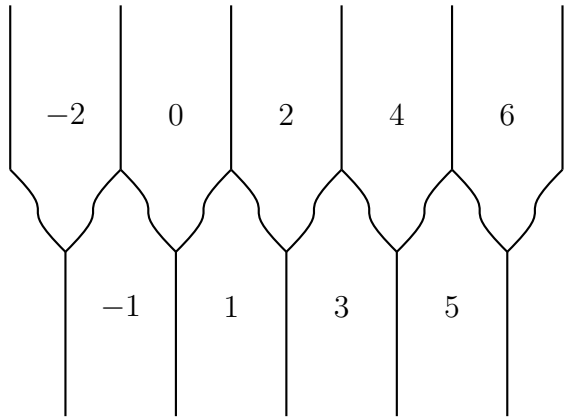


\centering
    
\psset{unit=4ex}
\pspicture(-2,-3.5)(10,5)

\multido{\n=0+2}{5}{
\psline(\n,0)(\n,-3)
\psline(!\n \space -1 add 1.5)(!\n \space -1 add 4.5)
\psbezier(\n,0)(!\n \space 1 add 1)(\n,0.5)(!\n \space 1 add 1.5)
\psbezier(\n,0)(!\n \space -1 add 1)(\n, 0.5)(!\n \space -1 add 1.5)
}
\psline(9,1.5)(9,4.5)

\multido{\n=-2+2}{5}{\rput(!\n \space 2 add 2.5){$\n$}}
\multido{\n=-1+2}{4}{\rput(!\n \space 2 add -1){$\n$}}

\endpspicture

\caption{Some level sets of the canonical bi-staircase. The separation between the regions with odd and even values is the graph of the function $\ft$.}
\label{fig:bi-staircase}
\end{figure}

\begin{rmk}[Heuristic interpretation]\label{rmk:bi-staircase}
\begin{em}

At this point one might ask what is the reason behind the rather mysterious definition (\ref{defn:ft}). In order to answer, let us consider, for example, the boundary between the region where $u=0$ and the region where $u=1$. Let us assume that in the rectangle $\Omega:=[0,1]\times[-R,R]$, with $R$ large enough, the frontier between the two regions is described by some curve $y=\ft(x)$. Then in $\Omega$ the functional $\JF$ with parameters given by (\ref{defn:abxi-simple}) is equal to
\begin{equation*}
    \alpha_\theta \int_0^1\sqrt{1+\ft'(x)^2}\,dx+ 
    3\int_0^1\left[ (\ft(x)+R) (1-x)^2 + (R-\ft(x))x^2 \right]\,dx.
\end{equation*}

If $u$ is a minimizer for $\JF$, then $\ft$ has to minimize this functional, and therefore if has to satisfy the Euler-Lagrange equation, that in this case reads as
\begin{equation*}
    \alpha_\theta\left(\frac{\ft'(x)}{\sqrt{1+\ft'(x)^2}}\right)'=
    3\left[(1-x)^2-x^2\right]=
    3-6x,
\end{equation*}
from which we obtain that
\begin{equation}
    \alpha_\theta\frac{\ft'(x)}{\sqrt{1+\ft'(x)^2}}=K+3x-3x^2
    \qquad
    \forall x\in[0,1]
    \label{eqn:ft'}
\end{equation}
for some real number $K$. In order to compute the value of $K$, we impose that the weighted sum of the three tangent vectors in the triple junction, corresponding to $x=0$, vanishes. In this sum the weight of the vertical vector, corresponding to the separation between 1 and $-1$, is $2^\theta$ times the weight of the other two vectors, corresponding to jump heights equal to~1. When we impose this condition we obtain that $K= \alpha_\theta\cdot 2^{\theta-1}$, so that the right-hand side of (\ref{eqn:ft'}) is exactly the function $\gt$ defined in (\ref{defn:gt}). At this point we compute $\ft'(x)$ from (\ref{eqn:ft'}) and we end up with (\ref{defn:ft}).

Incidentally, we observe here that 
\begin{equation}\label{eq:g=F1-F0}
    \gt(x)=\Ft(x,1)-\Ft(x,0)
    \qquad
    \forall x\in[0,1],
\end{equation}
where $\Ft$ is the function that we defined in (\ref{defn:calibration-1d}) in order to calibrate one dimensional staircases. This ``coincidence'' will be essential in the sequel.

\end{em}
\end{rmk}

\paragraph{\textit{\textmd{Step 3 -- The calibration method for the bi-staircase}}}

In the case of bi-staircases in $\re^2$ the calibration method reduces to the following.

\begin{prop}[Calibration for the bi-staircase]\label{prop:calibr-dd}
    For every real number $\theta\in[0,1)$, let $\alpha_\theta$ be defined as in (\ref{defn:abxi-simple}), and let us consider the function $\Ft$ defined in (\ref{defn:calibration-1d}). 
    
    Let us assume that there exists a continuous function $\At:[0,1]\times \R\to\re$ such that
    \begin{enumerate}
    \renewcommand{\labelenumi}{(\roman{enumi})}

        \item it satisfies the equalities
        \begin{equation}\label{eq:ext_cont}
        \At(0,z)=\At(0,-z)
        \qquad\mbox{and}\qquad 
        \At(1,z+2)=\At(1,-z)
        \qquad 
        \forall z\in\R;
        \end{equation}

        \item it satisfies the equality
        \begin{equation}
            \left[\At(x,1)-\At(x,0)\right]^2+\left[\Ft(x,1)-\Ft(x,0)\right]^2=
            \alpha_\theta ^2
            \qquad
            \forall x\in [0,1],
            \label{hp:ABV=}
        \end{equation}
        with
        \begin{equation}
            \At(x,1)>\At(x,0)
           \qquad
            \forall x\in [0,1];
            \label{hp:ABV-sign}
        \end{equation}
    
        \item it satisfies the inequality
        \begin{equation}
            \left[\At(x,z_2)-\At(x,z_1)\right]^2+\left[\Ft(x,z_2)-\Ft(x,z_1)\right]^2\leq
            \alpha_\theta ^2(z_2-z_1)^{2\theta}
            \label{hp:ABV<}
        \end{equation}
        for every $x\in[0,1]$ and every $z_1\leq z_2$.
            
    \end{enumerate} 
    
    Then the bi-staircase $\St(x)$ introduced in Definition~\ref{defn:bi-staircase}, along with all its oblique translations and their translations in the $y$-direction, is an entire local minimizer of the functional~(\ref{defn:JF-dd}), with parameters given by~(\ref{defn:abxi-simple})..

\end{prop}

\begin{proof}

As in the one-dimensional case, due to the invariance of the problem under ``oblique translations'' and translations in the $y$-direction, it is sufficient to prove that the canonical bi-staircase $\St$ is an entire local minimizer. The proof strategy consists in first extending $\At$ to the entire plane, and then using $\At$ and $\Ft$ to construct an auxiliary functional $\mathbb{G}(\Omega,v)$ for which a chain of equalities and inequalities of the form (\ref{osculating}) can be established.

\paragraph{\textit{\textmd{Extension}}}

Let us first extend $\At$ to the set $[-1,1]\times \R$ in an even way, by setting $\At(x,z) := \At(-x,-z)$ for every $x \in [-1,0]$. We then further extend $\At$ to the whole plane by $(2,2)$-periodicity; that is, we require that $\At(x+2,z+2) = \At(x,z)$ for every $(x,z) \in \R^2$. We remark that condition (\ref{eq:ext_cont}) ensures the consistency of this extension and guarantees that the resulting function is continuous on $\R^2$.

\paragraph{\textit{\textmd{Notation}}}

Let us introduce some notation. An open subset of the plane is called \emph{regular} if it is bounded and its boundary is piecewise $C^1$. For every regular open set $\Omega$, we call $PC(\Omega)$ the set of \emph{piecewise constant} functions in $\Omega$, namely the set of all functions $v\in\PJ(\Omega)$ whose image is finite, and whose level sets are (up to negligible sets) regular open sets. To every $v\in PC(\Omega)$ we associate three families of objects.
\begin{itemize}
    \item \emph{Level sets.} Let $\Imm(v)$ denote the image of $v$, which is a finite subset of $\re$. For every $z\in \Imm(v)$, we consider the differential form in the plane
    \begin{equation}
        \omega_z:=\At(x,z)\,dx+\Ft(x,z)\,dy,
        \label{defn:omega-z}
    \end{equation} 
    and the level set $\Omega_z:=v^{-1}(z)$. We observe that these level sets form a partition of $\Omega$, up to negligible sets.

    \item \emph{Parametrization of the boundary.} A finite collection of embedded curves $\{\gamma_j\}_{j\in J_1}$ of class $C^1$, each of which parametrizes a portion of the boundary $\partial\Omega$, and with the property that any two distinct curves intersect at most at their endpoints. We assume that, for every $j\in J_1$, the curve $\gamma_j$ is oriented in such a way that $\Omega$ lies ``on the left'' of $\gamma_j$, and the function $v$ takes a constant value $z_{\ell,j} \in \operatorname{Im}(v)$ on the left side of $\gamma_j$.

    \item  \emph{Parametrization of the jump set.} A finite collection of embedded curves $\{\gamma_j\}_{j\in J_2}$ of class $C^1$, each of which parametrizes a portion of the boundary of the level sets of $v$ (not already included in the boundary of $\Omega$), and hence describes the jump set of $v$. Again any two distinct curves can intersect at most at their endpoints. For each $j \in J_2$, we assume that $\gamma_j$ separates the level set corresponding to some value $z_{\ell,j} \in \operatorname{Im}(v)$ on its left from the level set corresponding to some value $z_{r,j} \in \operatorname{Im}(v)$ on its right.
    
\end{itemize}

\paragraph{\textit{\textmd{The osculating functional and its properties}}}

For every regular open set $\Omega$, and every $v\in PC(\Omega)$, we define the functional
\begin{equation}
    \mathbb{G}(\Omega,v):=\sum_{j\in J_1}\int_{\gamma_j}\omega_{z_{\ell,j}}.
    \label{defn:G-dd}
\end{equation}

The key properties of this functional are the following.
\begin{enumerate}
\renewcommand{\labelenumi}{(\arabic{enumi})}

    \item (Dependence only on the values at the boundary). If $v$ and $w$ are elements of $PC(\Omega)$ that coincide in a neighborhood of $\partial\Omega$, then
    \begin{equation*}
        \mathbb{G}(\Omega,v)=\mathbb{G}(\Omega,w).
    \end{equation*}

    \item (Representation in terms of levels sets and jump set). We can rewrite (\ref{defn:G-dd}) as
    \begin{equation}
        \mathbb{G}(\Omega,v)=
        \sum_{z\in \Imm(v)}\int_{\partial\Omega_z}\omega_z+
        \sum_{j\in J_2}\int_{\gamma_j}\left(\omega_{z_{r,j}}-\omega_{z_{\ell,j}}\right),
        \label{eqn:G1}
    \end{equation}
    where, in the integrals of the first sum, the boundary of $\Omega_z$ is oriented as usual in such a way that $\Omega_z$ lies on the left of the boundary. In the sequel we refer to the terms of first sum as the contribution of level sets, and to the terms of the second sum as the contribution of the jumps.

    \item  (Inequality and equality for level sets). For every $z\in \Imm(v)$ it turns out that
    \begin{equation}
        \int_{\partial\Omega_z}\omega_z\leq
        3\int_{\Omega_z}(z-x)^2\,dx\,dy,
        \label{th:calibr-fidelity}
    \end{equation}
    with equality if
    \begin{equation}
        \Omega_z\subseteq\{(x,z)\in\R^2:|z-x|\leq \sigma_\theta\}.
        \label{Hp:Omega-z}
    \end{equation}

    \item (Inequality and equality for jump terms). For every $j\in J_2$ it turns out that
    \begin{equation}
        \int_{\gamma_j}\left(\omega_{z_{r,j}}-\omega_{z_{\ell,j}}\right)\leq
        \alpha_\theta|z_{r,j}-z_{\ell,j}|^\theta\cdot\operatorname{Length}(\gamma_j),
        \label{th:calib-boundary}
    \end{equation}
    with equality if $\gamma_j$ is contained in the jump set of $\St$ and $(z_{\ell,j},z_{r,j})$ are the two values of $\St$ on the two sides of $\gamma_j$.

\end{enumerate}

\paragraph{\textit{\textmd{Proof of properties (1), (2), (3) of the osculating functional}}}

Property~(1) is an immediate consequence of definition (\ref{defn:G-dd}), which depends only on the values of $v$ in a neighborhood of $\partial\Omega$.

In order to show property~(2), by adding and subtracting equal terms, we write (\ref{defn:G-dd}) as
\begin{equation}
    \mathbb{G}(\Omega,v)=\sum_{j\in J_1}\int_{\gamma_j}\omega_{z_{\ell,j}}+
    \sum_{j\in J_2}\int_{\gamma_j}\left(\omega_{z_{\ell,j}}-\omega_{z_{r,j}}\right)+
    \sum_{j\in J_2}\int_{\gamma_j}\left(\omega_{z_{r,j}}-\omega_{z_{\ell,j}}\right).
    \label{eqn:G2}
\end{equation}

The goal is showing that the first two sums in (\ref{eqn:G2}) are equal to the first sum in (\ref{eqn:G1}). With this goal in mind, for every $z\in \Imm(v)$ we group together all the terms in the first two sums of (\ref{eqn:G2}) that represent the integral of $\omega_z$ along pieces of $\partial\Omega_z$. We observe that all these terms appear with the sign corresponding to the correct orientation of $\partial\Omega_z$ since, for every $j\in J_2$, the curve $\gamma_j$ has the correct orientation when considered as a part of the boundary of $\Omega_{z_{\ell,j}}$, and the reverse orientation (which compensate the negative sign) when considered as a part of the boundary of $\Omega_{z_{r,j}}$.

In order to show property~(3), we observe that $\At$ does not depend on $y$, and hence from Green's theorem (ore more generally from Stokes' theorem) we deduce that
\begin{equation*}
        \int_{\partial\Omega_z}\omega_z=
        \int_{\Omega_z}\frac{\partial\Ft}{\partial x}(x,z)\,dx\,dy,
\end{equation*}
    so that the inequality (\ref{th:calibr-fidelity}) follows from estimate (\ref{est:dF/dx<}). 
    
On the other hand, when $\Omega_z$ satisfies (\ref{Hp:Omega-z}), we know that there is equality in (\ref{est:dF/dx<}), and this yields equality in (\ref{th:calibr-fidelity}).

\paragraph{\textit{\textmd{Proof of property (4) of the osculating functional}}}

To begin with, we write the components of $\gamma_j$ as
\begin{equation*}
    \gamma_j(t)=(x_j(t),y_j(t))
    \qquad
    t\in[a_j,b_j],
\end{equation*}
and we obtain that
\begin{equation}
    \int_{\gamma_j}\left(\omega_{z_{r,j}}-\omega_{z_{\ell,j}}\right)=
    \int_{a_j}^{b_j}\langle E_j(t),\gamma_j'(t)\rangle\,dt,
    \label{eqn:int-gamma}
\end{equation}
where
\begin{equation*}
    E_j(t):=\left(
    \At(x_j(t),z_{r,j})-\At(x_j(t),z_{\ell,j}),
    \Ft(x_j(t),z_{r,j})-\Ft(x_j(t),z_{\ell,j})
    \strut\right).
\end{equation*}

Now from  assumption (\ref{hp:ABV<}) we deduce that
\begin{equation*}
    \|E_j(t)\|\leq\alpha_\theta|z_{r,j}-z_{\ell,j}|^\theta
    \qquad
    \forall t\in[a_j,b_j],
\end{equation*}
and therefore from Cauchy--Schwarz inequality we conclude that
    \begin{equation*}
        \int_{\gamma_j}\left(\omega_{z_{r,j}}-\omega_{z_{\ell,j}}\right)\leq
        \int_{a_j}^{b_j}\|E_j(t)\|\cdot\|\gamma_j'(t)\|\,dt\leq
        \alpha_\theta|z_{r,j}-z_{\ell,j}|^\theta
        \int_{a_j}^{b_j}\|\gamma_j'(t)\|\,dt,
    \end{equation*}
which proves the inequality in (\ref{th:calib-boundary}).

If $\gamma_j$ is contained in the jump set of $\St$ we distinguish two cases.
    \begin{itemize}
        \item Let us assume that, for some integer $k$, the curve $\gamma_j$ is contained in the part of the vertical line $x=k$ that separates the values $k-1$ and $k+1$. In this case, up to a reparametrization of the curve, we can assume that $\gamma_j(t)=(k,t)$, with $t$ in some interval $[a_j,b_j]$, while $z_{r,j}=k+1$ and $z_{\ell,j}=k-1$. 
        
        Under these assumptions, from (\ref{defn:calibration-1d}) and (\ref{defn:cubic}) we obtain that
        \begin{equation*}
            \langle E_j(t),\gamma_j'(t)\rangle=
            \Ft(k,k+1)-\Ft(k,k-1)=
            \alpha_\theta\cdot 2^\theta,
        \end{equation*}
        and therefore from (\ref{eqn:int-gamma}) we deduce the equality in (\ref{th:calib-boundary}).

        \item  Let us assume that, for some integer $k$, the curve $\gamma_j$ is contained in the increasing branch of $\ft$ that separates the values $2k$ and $2k+1$ (the case where it is contained in the decreasing branch that separates $2k-1$ and $2k$ is analogous because of the symmetries of $\At$ and $\ft$). In this case, up to a reparametrization of the curve, we can assume that $\gamma_j(t)=(t,\ft(t))$, with $t$ in some interval $[a_j,b_j]\subseteq [2k,2k+1]$, while $z_{r,j}=2k+1$ and $z_{\ell,j}=2k$. 
        
        Under these assumptions, from the 2-periodicity of $\ft$ and (\ref{defn:ft}) we obtain that
        \begin{equation*}
            \gamma_j'(t)=(1,\ft'(t))=(1,\ft'(t-2k))=
            \left(1,\frac{\gt(t-2k)}{\sqrt{\alpha_\theta ^2 -\gt(t-2k)^2}}\right).
        \end{equation*}

        From the definition of $\Ft$, and (\ref{eq:g=F1-F0}) with $x=t-2k$, we obtain that
        \begin{equation*}
            \Ft(t,2k+1)-\Ft(t,2k)=
            \Ft(t-2k,1)-\Ft(t-2k,0)=
            \gt(t-2k).
        \end{equation*}
        
        From the $(2,2)$-periodicity of $\At$ it follows that
        \begin{equation*}
            \At(t,2k+1)-\At(t,2k)=
            \At(t-2k,1)-\At(t-2k,0),
        \end{equation*}
        and hence, from (\ref{hp:ABV=}) and (\ref{hp:ABV-sign}) with $x=t-2k$, we obtain that
        \begin{equation*}
            \At(t,2k+1)-\At(t,2k)=\sqrt{\alpha_\theta ^2-\gt(t-2k)^2}.
        \end{equation*}

        From all these relations we deduce that
         \begin{equation*}
            \langle E_j(t),\gamma_j'(t)\rangle=
            \sqrt{\alpha_\theta ^2-\gt(t-2k)^2}+\frac{\gt(t-2k)^2}{\sqrt{\alpha_\theta ^2-\gt(t-2k)^2}}=
            \alpha_\theta\|\gamma_j'(t)\|.
        \end{equation*}

        Plugging this equality into (\ref{eqn:int-gamma}), and recalling that now $|z_{r,j}-z_{\ell,j}|=1$, we conclude that we have equality in (\ref{th:calib-boundary}).
    \end{itemize}

\paragraph{\textit{\textmd{Inequality for piecewise constant competitors}}}

Let $\Omega$ be a regular open set such that the canonical bi-staircase $\St$ belongs to $PC(\Omega)$. We claim that
\begin{equation*}
    \JF(\Omega, v) \geq
    \mathbb{G}(\Omega, v) =
    \mathbb{G}(\Omega, \St) =
    \JF(\Omega, \St)
\end{equation*}
for every $v \in PC(\Omega)$ that coincides with $\St$ in a neighborhood of $\partial\Omega$, where $\JF$ denotes the functional~(\ref{defn:JF-dd}) with parameters given by~(\ref{defn:abxi-simple}). Indeed:
\begin{itemize}
    \item the first inequality follows from the representation~(\ref{eqn:G1}), since the inequalities~(\ref{th:calibr-fidelity}) and~(\ref{th:calib-boundary}) imply that each term in~(\ref{eqn:G1}) is less than or equal to the corresponding contribution to $\JF(\Omega, v)$ from the relevant level or jump set;

    \item the middle equality holds because $\St$ and $v$ coincide in a neighborhood of $\partial\Omega$;

    \item the final equality follows again from the representation~(\ref{eqn:G1}), now applied to the function $\St$, since in this case equalities hold in both~(\ref{th:calibr-fidelity}) and~(\ref{th:calib-boundary}).
\end{itemize}

This suffices to show that $\St$ is a minimizer among all piecewise constant competitors in the regular open set $\Omega$.

\paragraph{\textit{\textmd{Inequality for general competitors}}}

It remains to consider a generic function $v \in \PJ(\Omega)$ that coincides with $\St$ in a neighborhood of $\partial\Omega$. In this case, one can approximate $v$ by a sequence $\{v_n\} \subseteq PC(\Omega)$ such that each $v_n$ coincides with $\St$ in a neighborhood of $\partial\Omega$, and such that $v_n \to v$ in $L^2(\Omega)$ and
\begin{equation*}
    \lim_{n \to +\infty} \int_{S_{v_n}} |v_n^+(x) - v_n^-(x)|^\theta\, d\mathcal{H}^1 =
    \int_{S_v} |v^+(x) - v^-(x)|^\theta\, d\mathcal{H}^1,
\end{equation*}
so that in particular $\JF(\Omega, v_n) \to \JF(\Omega, v)$ as $n \to +\infty$.

Approximation results of this type, without the equality condition near the boundary, can be found in~\cite{2014-M3AS-BelChaGol} for the case $\theta \in (0,1)$, and in~\cite{density-partitions} for the case $\theta = 0$. However, in both cases the constructions can be adapted to produce a sequence that satisfies this additional condition. This is particularly feasible in our setting, since we know that $v$ is well-behaved in a neighborhood of the boundary of $\Omega$, where it coincides with $\St$.

At this point, the conclusion is straightforward. From the previous paragraph, we know that $\JF(\Omega, v_n) \geq \JF(\Omega, \St)$ for every $n$. Passing to the limit, we obtain 
$$\JF(\Omega, v) \geq \JF(\Omega, \St).$$ 

Since regular open sets where $\St$ is piecewise constant exhaust the whole plane, this completes the proof.
\end{proof}


\begin{rmk}[Comparison with~\cite{2003-CalcVar-ABDM}]\label{rmk:ABDM-2d}
\begin{em}

We point out that in Proposition~\ref{prop:calibr-dd} we did not assume any differentiability of $\At$, and its continuity was used only to give meaning to the line integrals of the form~(\ref{defn:omega-z}). If $\At$ is more regular, one can define a vector field $\Phi$ as the curl of $(\At, \Ft, 0)$, namely
\begin{equation*}
\Phi(x,y,z) := \left(
-\frac{\partial \Ft}{\partial z}(x,z), 
\frac{\partial \At}{\partial z}(x,z), 
\frac{\partial \Ft}{\partial x}(x,z)
\right).
\end{equation*}

In this way, the terms in the first sum of~(\ref{eqn:G1}) can be interpreted as the flux of $\Phi$ across the horizontal parts of the graph of $v$, while the terms in the second sum correspond to the flux across the ``vertical parts'' of the graph, associated with jumps. In other words, in this case $\mathbb{G}(\Omega,v)$ coincides with the null Lagrangian introduced in~\cite{2003-CalcVar-ABDM}.

\end{em}    
\end{rmk}

\paragraph{\textit{\textmd{Step 4 -- Construction of a special piecewise affine function}}}

\begin{lemma}\label{lemma:psi}

Let $-2\leq c\leq 0\leq d\leq 2$ be two real numbers, with $c<d$. Let us set
\begin{equation}
    C:=\sqrt{16-(c-2)^2},
    \qquad\qquad
    D:=\sqrt{16-(d-c)^2}-\sqrt{16-(2-c)^2},
    \label{defn:CD}
\end{equation}
and let $\psi:\re\to\re$ be the function such that
\begin{enumerate}
\renewcommand{\labelenumi}{(\roman{enumi})}
    \item $\psi(c)=-C$ and $\psi(d)=D$,

    \item $\psi(\sigma)=0$ for every $\sigma\leq -2$ and for every $\sigma\geq 2$,

    \item $\psi$ is an affine function in each of the intervals $[-2,c]$, $[c,d]$, and $[d,2]$ (the first and last interval might be a single point).
\end{enumerate}

Then the function $\psi$ satisfies the equality
\begin{equation}
    (\psi(d)-\psi(c))^2+(d-c)^2=16
    \qquad\quad\text{with}\quad\qquad
    \psi(d)>\psi(c),
    \label{th:psi=}
\end{equation}
and the inequality
\begin{equation}
    (\psi(\sigma_2)-\psi(\sigma_1))^2+(\sigma_2-\sigma_1)^2\leq16
    \qquad
    \forall (\sigma_1,\sigma_2)\in[-2,2]^2.
    \label{psi<=}
\end{equation}
    
\end{lemma}

\begin{proof}

The verification of (\ref{th:psi=}) is immediate from (\ref{defn:CD}). So let us concentrate on the inequality (\ref{psi<=}). Due to the piecewise definition of $\psi$, we should a priori consider nine cases according to the position of $\sigma_1$ and $\sigma_2$ with respect to $c$ and $d$. On the other hand, since $\psi$ is a piecewise affine function, the left-hand side of (\ref{psi<=}) is a convex function of both $\sigma_1$ and $\sigma_2$ in each of the intervals $[-2,c]$, $[c,d]$, and $[d,2]$. As a consequence, we can reduce ourselves to check the inequality when $\sigma_1$ and $\sigma_2$ are endpoints of these intervals, and therefore we are left with the six cases shown in the following table.

\begin{center}
\renewcommand{\arraystretch}{1.3}
\begin{tabular}{|c||c|c|c|}
\hline
Case & $\sigma_1$ & $\sigma_2$ & Inequality to check
\\
\hline\hline
1 & $-2$ & $c$ & $C^2+(c+2)^2\leq16$
\\
\hline
2 & $-2$ & $d$ & $D^2+(d+2)^2\leq16$
\\
\hline
3 & $-2$ & $2$ & $16\leq16$
\\
\hline
4 & $c$ & $d$ & $(D+C)^2+(d-c)^2\leq16$
\\
\hline
5 & $c$ & $2$ & $C^2+(2-c)^2\leq16$
\\
\hline
6 & $d$ & $2$ & $D^2+(2-d)^2\leq16$
\\
\hline
\end{tabular}

\end{center}

Let us examine the six cases.
\begin{itemize}
    \item The inequality of case~3 is immediate.

    \item  Since $c\leq 0$, the inequality of case~1 is satisfied whenever the inequality of case~5 is satisfied, and the latter is an equality due to the definition of $C$ in (\ref{defn:CD}).

    \item  The inequality of case~4 is an equality because of (\ref{defn:CD}).

    \item  Since $d\geq 0$, the inequality of case~6 is satisfied whenever the inequality of case~2 is satisfied.
\end{itemize}

As a consequence, we only need to prove the inequality of case~2. Taking (\ref{defn:CD}) into account, with some algebra this inequality reduces to
\begin{equation}
    16+(d+2)^2\leq
    (d-c)^2+(2-c)^2+
    2\sqrt{16-(d-c)^2}\cdot\sqrt{16-(2-c)^2}.
    \label{est:lemma-main}
\end{equation}

Let us consider the right-hand side as a function of $c$. For every $c\in(-2,d)$ its derivative with respect to $c$ is equal to
\begin{equation*}
    2\left(\sqrt{16-(d-c)^2}-\sqrt{16-(2-c)^2}\right)
    \left(\frac{2-c}{\sqrt{16-(2-c)^2}}-
    \frac{d-c}{\sqrt{16-(d-c)^2}}\right).
\end{equation*}

Now we observe that 
$0<d-c\leq 2-c< 4$
for every $c\in(-2,0)$, and the function $x\mapsto x/\sqrt{16-x^2}$ is increasing in $(0,4)$. It follows that the derivative is positive for every $c\in(-2,0)$, and hence the right-hand side of (\ref{est:lemma-main}), as a function of $c$, is increasing in the interval $[-2,0]$. As a consequence, it is enough to check (\ref{est:lemma-main}) when $c=-2$, in which case it reduces to an equality.
\end{proof}

\paragraph{\textit{\textmd{Step 5 -- Construction of the calibration for $\theta=0$}}}

Let us specialize now to the case $\theta=0$. In this case from (\ref{defn:abxi-simple}) we get $\alpha_0=4$, while from (\ref{defn:cubic-true}) and (\ref{defn:cubic}) we obtain that
\begin{equation*}
    \phihat_0(x)=\varphi_0(x)=3x-x^3
    \qquad
    \forall x\in[-1,1],
\end{equation*}
so that in particular
\begin{equation*}
    \varphi_0(-x)=x^3-3x
    \qquad\quad\text{and}\quad\qquad
    \varphi_0(1-x)=x^3-3x^2+2
\end{equation*}
for every $x\in[0,1]$. It follows that
$$-2\leq\varphi_0(-x)\leq 0\leq \varphi_0(1-x)\leq 2\qquad \mbox{and}\qquad \varphi_0(-x)<\varphi_0(1-x),$$
and hence we can apply Lemma~\ref{lemma:psi} with
\begin{equation}\label{defn:cd}
    c:=\varphi_0(-x)
    \qquad\quad\text{and}\quad\qquad
    d:=\varphi_0(1-x).
\end{equation}

We obtain a piecewise affine function that now we denote by $\psi(x,\sigma)$ in order to highlight that it depends also on $x$. Finally, we set
\begin{equation*}
    A_0(x,z):=\psi(x,\phihat_0(z-x))
    \qquad
    \forall x\in[0,1],
    \quad
    \forall z\in\re.
\end{equation*}

We claim that this function satisfies the assumptions of Proposition~\ref{prop:calibr-dd}, and therefore it is exactly what we need for the calibration method.

\begin{itemize}
    \item We check that this function satisfies (\ref{hp:ABV=}), (\ref{hp:ABV-sign}) and (\ref{hp:ABV<}) for $\theta=0$. To this end, it is enough to observe that the function $F_0$ defined according to (\ref{defn:calibration-1d}) satisfies
\begin{equation*}
    F_0(x,z_2)-F_0(x,z_1)=
    \phihat_0(z_2-x)-\phihat_0(z_1-x)
\end{equation*}
for every $x\in[0,1]$ and every pair $(z_1,z_2)\in\re^2$, and therefore the three required relations are exactly the three properties of the function $\psi(x,\sigma)$ provided by Lemma~\ref{lemma:psi}.

\item  The continuity of $A_0$ follows from the continuity of $\phihat_0$, the continuity of $c$ and $d$ in (\ref{defn:cd}) with respect to $x$, and the continuity of $\psi$ with respect to $c$, $d$, and $\sigma$. The latter follows from the piecewise affine definition of $\psi$ in Lemma~\ref{lemma:psi}, with particular attention to the cases where $c \to -2^+$ and/or $d \to 2^-$, in which case the corresponding values $C$ and $D$ tend to $0$.

\item Finally, we observe that for $x = 0$ we have $c = 0$ and $d = 2$, while for $x = 1$ we have $c = -2$ and $d = 0$. In both cases, one of the three intervals where $\psi$ is affine degenerates to a single point, while the other two intervals are $[-2,0]$ and $[0,2]$. As a consequence, the functions $\sigma \mapsto \psi(0,\sigma)$ and $\sigma \mapsto \psi(1,\sigma)$ are even. Since $\phihat_0$ is an odd function, it follows that
\begin{equation*}
    A_0(0,z)=
    \psi(0,\phihat_0(z))=
    \psi(0,\phihat_0(-z))=
    A_0(0,-z),
\end{equation*}
which proves the first equality in (\ref{eq:ext_cont}). Analogously, we obtain that
\begin{equation*}
    A_0(1,-z)= 
    \psi(1,\phihat_0(-z-1))=
    \psi(1,\phihat_0(z+1))=
    A_0(1,z+2),
\end{equation*}
which proves the second equality in (\ref{eq:ext_cont}).

\end{itemize}

\begin{rmk}[Dimension one vs dimension two]
\begin{em}

Let us highlight a notable difference between our calibrations in one and two dimensions. In the one-dimensional case, the function $\Ft$ that we constructed was sufficient to calibrate simultaneously all oblique translations of the canonical staircase. In contrast, the calibration we defined in two dimensions successfully calibrates the canonical staircase and all its oblique translations (where the component $A_0$ plays no role), as well as the bi-staircase of Definition~\ref{defn:bi-staircase} and all its translations in the $y$-direction. However, it does \emph{not} calibrate oblique translations of the bi-staircase. We suspect that it is not possible to calibrate in the same time all oblique translations of the bi-staircase.

\end{em}
\end{rmk}

\section{Open problems}\label{sec:open}

In this final section we mention some open problems concerning entire local minimizers for (\ref{defn:JF-dd}).

\begin{open}[More general exponents]

    Extend statement~(2) of Theorem~\ref{thm:main-dd} to all exponents $\theta\in(0,1)$.
    
\end{open}

To this end, it would suffice to show that the canonical bi-staircase introduced in Definition~\ref{defn:bi-staircase} is an entire local minimizer also for $\theta \in (0,1)$. This, in turn, would follow from the existence of a function $\At$ satisfying the assumptions of Proposition~\ref{prop:calibr-dd}. The construction we provided for $\theta = 0$ can be extended to more general values $\theta \in (0,1)$, but the proof is simpler when $\theta = 0$ because the right-hand side of~(\ref{hp:ABV<}) reduces to a constant, allowing us to exploit the convexity argument used in the proof of Lemma~\ref{lemma:psi}. Some basic numerical experiments seem to suggest that the construction works also for other values of $\theta$.

The second open problem concerns the existence of less symmetric entire local minimizers. This question is motivated by the fact that the canonical bi-staircase of Definition~\ref{defn:bi-staircase} retains a certain degree of symmetry, as is evident from Figure~\ref{fig:bi-staircase}.

\begin{open}[Asymmetric exotic minimizers]

    Determine whether there exists an entire local minimizer for (\ref{defn:JF-dd}), with parameters as in (\ref{defn:abxi-simple}), that is asymptotic for $y\to +\infty$ to the staircase with values $2k$, and for $y\to -\infty$ to the staircase with values $2k+\tau_0$, for some $\tau_0\in(-1,1)$.
    
\end{open}

All the considerations of Remark~\ref{rmk:bi-staircase} still apply, even in the case of asymmetric minimizers. More precisely, the curve that separates the region with value $a$ above from the region with value $b$ below is the graph of a function $f_\theta$ that satisfies
\begin{equation*}
\alpha_\theta |b-a|^\theta \left(\frac{f_\theta'(x)}{\sqrt{1 + f_\theta'(x)^2}}\right)' =
3\left[(b - x)^2 - (a - x)^2\right],
\end{equation*}
and, once again, the values of $f_\theta'$ at the endpoints must be chosen so that the weighted sum of the three tangent vectors at each triple point vanishes. We observe, however, that in this less symmetric setting there is no reason why the separation between the values $c$ and $c+2$, in either the upper or lower region, should be a half-line.

The next open question addresses the full characterization of entire local minimizers.

\begin{open}[Characterization of local minimizers]

    Find the set of all entire local minimizer for (\ref{defn:JF-dd}).
    
\end{open}

In Theorem~\ref{thm:main-1d}, we answered the corresponding question in the one-dimensional case, but the arguments used in the proof appear to be quite specific to dimension one. In higher dimensions, we are not able to answer this question, which seems to be nontrivial even in the case where the forcing term is identically zero (cf.~Remark~\ref{rmk:M=0}). Incidentally, this highlights a drawback of the calibration method: while it is often a powerful tool for verifying that a given candidate is a minimizer, it appears ineffective for ruling out the existence of alternative minimizers.

As explained in the introduction, characterizing all entire local minimizers of (\ref{defn:JF-dd}) was the original motivation for this research, as the problem arises naturally in the study of the asymptotic behavior of minimizers for certain regularizations of the Perona–Malik functional. Our initial conjecture was that the only minimizers were standard simple staircases, but we now know this is not the case. For this reason, we conclude the paper by asking whether our exotic minimizers have any relevance for the models that originally motivated this work.

\begin{open}[Back to the Perona-Malik functional]

    Determine whether exotic entire local minimizers can emerge as limits of blow-ups of minimizers for the higher dimensional versions of (\ref{defn:SPMF}) or (\ref{defn:DPMF}).
    
\end{open}


\subsubsection*{\centering Acknowledgments}
The authors are members of the {\selectlanguage{italian} ``Gruppo Nazionale per l'Analisi Matematica, la Probabilità e le loro Applicazioni''} (GNAMPA) of the {\selectlanguage{italian}``Istituto Nazionale di Alta Matematica''} (INdAM).
The authors acknowledge the MIUR Excellence Department Project awarded to the Department of Mathematics, University of Pisa, CUP I57G22000700001.

%
%
%
%


\begin{thebibliography}{10}
\providecommand{\url}[1]{\texttt{#1}}
\providecommand{\urlprefix}{URL }
\providecommand{\selectlanguage}[1]{\relax}
\providecommand{\eprint}[2][]{\url{#2}}

\bibitem{2003-CalcVar-ABDM}
\textsc{G.~Alberti}, \textsc{G.~Bouchitt\'{e}}, \textsc{G.~Dal~Maso}.
\newblock The calibration method for the {M}umford-{S}hah functional and
  free-discontinuity problems.
\newblock \emph{Calc. Var. Partial Differential Equations} \textbf{16} (2003),
  no.~3, 299--333.

\bibitem{AFP}
\textsc{L.~Ambrosio}, \textsc{N.~Fusco}, \textsc{D.~Pallara}.
\newblock \emph{Functions of bounded variation and free discontinuity
  problems}.
\newblock Oxford Mathematical Monographs. The Clarendon Press, Oxford
  University Press, New York, 2000.

\bibitem{frac_CKN1}
\textsc{W.~Ao}, \textsc{A.~DelaTorre}, \textsc{M.~d.~M. Gonz\'alez}.
\newblock Symmetry and symmetry breaking for the fractional
  {C}affarelli-{K}ohn-{N}irenberg inequality.
\newblock \emph{J. Funct. Anal.} \textbf{282} (2022), no.~11, Paper No. 109438,
  58.

\bibitem{2014-M3AS-BelChaGol}
\textsc{G.~Bellettini}, \textsc{A.~Chambolle}, \textsc{M.~Goldman}.
\newblock The {$\Gamma$}-limit for singularly perturbed functionals of
  {P}erona-{M}alik type in arbitrary dimension.
\newblock \emph{Math. Models Methods Appl. Sci.} \textbf{24} (2014), no.~6,
  1091--1113.

\bibitem{density-partitions}
\textsc{A.~Braides}, \textsc{S.~Conti}, \textsc{A.~Garroni}.
\newblock Density of polyhedral partitions.
\newblock \emph{Calc. Var. Partial Differential Equations} \textbf{56} (2017),
  no.~2, Paper No. 28, 10.

\bibitem{Symm-Newton}
\textsc{F.~Brock}, \textsc{V.~Ferone}, \textsc{B.~Kawohl}.
\newblock A symmetry problem in the calculus of variations.
\newblock \emph{Calc. Var. Partial Differential Equations} \textbf{4} (1996),
  no.~6, 593--599.

\bibitem{CKN}
\textsc{L.~Caffarelli}, \textsc{R.~Kohn}, \textsc{L.~Nirenberg}.
\newblock First order interpolation inequalities with weights.
\newblock \emph{Compositio Math.} \textbf{53} (1984), no.~3, 259--275.

\bibitem{CKN1}
\textsc{F.~Catrina}, \textsc{Z.-Q. Wang}.
\newblock On the {C}affarelli-{K}ohn-{N}irenberg inequalities: sharp constants,
  existence (and nonexistence), and symmetry of extremal functions.
\newblock \emph{Comm. Pure Appl. Math.} \textbf{54} (2001), no.~2, 229--258.

\bibitem{DGS-Wirtinger}
\textsc{B.~Dacorogna}, \textsc{W.~Gangbo}, \textsc{N.~Sub\'ia}.
\newblock Sur une g\'en\'eralisation de l'in\'egalit\'e{} de {W}irtinger.
\newblock \emph{Ann. Inst. H. Poincar\'e{} C Anal. Non Lin\'eaire} \textbf{9}
  (1992), no.~1, 29--50.

\bibitem{DKR-JFA}
\textsc{S.~Daneri}, \textsc{A.~Kerschbaum}, \textsc{E.~Runa}.
\newblock One-dimensionality of the minimizers for a diffuse interface
  generalized antiferromagnetic model in general dimension.
\newblock \emph{J. Funct. Anal.} \textbf{283} (2022), no.~12, Paper No. 109715,
  53.

\bibitem{DR-ARMA}
\textsc{S.~Daneri}, \textsc{E.~Runa}.
\newblock Exact periodic stripes for minimizers of a local/nonlocal interaction
  functional in general dimension.
\newblock \emph{Arch. Ration. Mech. Anal.} \textbf{231} (2019), no.~1,
  519--589.

\bibitem{DR-SIAM}
\textsc{S.~Daneri}, \textsc{E.~Runa}.
\newblock Pattern formation for a local/nonlocal interaction functional arising
  in colloidal systems.
\newblock \emph{SIAM J. Math. Anal.} \textbf{52} (2020), no.~3, 2531--2560.

\bibitem{DR-CalcVar}
\textsc{S.~Daneri}, \textsc{E.~Runa}.
\newblock One-dimensionality of the minimizers in the large volume limit for a
  diffuse interface attractive/repulsive model in general dimension.
\newblock \emph{Calc. Var. Partial Differential Equations} \textbf{61} (2022),
  no.~1, Paper No. 12, 31.

\bibitem{frac_CKN2}
\textsc{N.~De~Nitti}, \textsc{F.~Glaudo}, \textsc{T.~König}.
\newblock Non-degeneracy, stability and symmetry for the fractional
  {C}affarelli-{K}ohn-{N}irenberg inequality.
\newblock ArXiv:2403.02303 (2024).

\bibitem{CKN2}
\textsc{J.~Dolbeault}, \textsc{M.~J. Esteban}, \textsc{M.~Loss}.
\newblock Rigidity versus symmetry breaking via nonlinear flows on cylinders
  and {E}uclidean spaces.
\newblock \emph{Invent. Math.} \textbf{206} (2016), no.~2, 397--440.

\bibitem{nonsymm-gs}
\textsc{M.~J. Esteban}.
\newblock Nonsymmetric ground states of symmetric variational problems.
\newblock \emph{Comm. Pure Appl. Math.} \textbf{44} (1991), no.~2, 259--274.

\bibitem{Federer-book}
\textsc{H.~Federer}.
\newblock \emph{Geometric measure theory}, \emph{Die Grundlehren der
  mathematischen Wissenschaften}, volume Band 153.
\newblock Springer-Verlag New York, Inc., New York, 1969.

\bibitem{GGR-COCV-2018}
\textsc{M.~Ghisi}, \textsc{M.~Gobbino}, \textsc{G.~Rovellini}.
\newblock Symmetry-breaking in a generalized {W}irtinger inequality.
\newblock \emph{ESAIM Control Optim. Calc. Var.} \textbf{24} (2018), no.~4,
  1381--1394.

\bibitem{FastPM-CdV}
\textsc{M.~Gobbino}, \textsc{N.~Picenni}.
\newblock A quantitative variational analysis of the staircasing phenomenon for
  a second order regularization of the {P}erona-{M}alik functional.
\newblock \emph{Trans. Amer. Math. Soc.} \textbf{376} (2023), no.~8,
  5307--5375.

\bibitem{FastPM-CdV2}
\textsc{M.~Gobbino}, \textsc{N.~Picenni}.
\newblock Multi-scale analysis of minimizers for a second order regularization
  of the {P}erona--{M}alik functional.
\newblock \emph{Calc. Var. Partial Differential Equations} \textbf{64} (2025),
  no.~5, Paper No. 156.

\bibitem{GR-CalcVar}
\textsc{M.~Goldman}, \textsc{E.~Runa}.
\newblock On the optimality of stripes in a variational model with non-local
  interactions.
\newblock \emph{Calc. Var. Partial Differential Equations} \textbf{58} (2019),
  no.~3, Paper No. 103, 26.

\bibitem{Kawohl-symmetry_or_not}
\textsc{B.~Kawohl}.
\newblock Symmetry or not?
\newblock \emph{Math. Intelligencer} \textbf{20} (1998), no.~2, 16--22.

\bibitem{PicNic:PhD}
\textsc{N.~Picenni}.
\newblock \emph{The Perona-Malik problem: singular perturbation and
  semi-discrete approximation}.
\newblock Ph.D. thesis, Scuola Normale Superiore, 2023.
\newblock {h}ttps://dx.doi.org/10.25429/picenni-nicola\_phd2023-11-23.

\bibitem{fastpm-discreto}
\textsc{N.~Picenni}.
\newblock Staircasing effect for minimizers of the one-dimensional discrete
  {P}erona--{M}alik functional.
\newblock \emph{ESAIM Control Optim. Calc. Var.} \textbf{30} (2024), Paper No.
  44, 45.

\bibitem{PP-BLMS-2022}
\textsc{A.~Pluda}, \textsc{M.~Pozzetta}.
\newblock Minimizing properties of networks via global and local calibrations.
\newblock \emph{Bull. Lond. Math. Soc.} \textbf{55} (2023), no.~6, 3029--3052.

\bibitem{Solci}
\textsc{M.~Solci}.
\newblock Local interpolation techniques for higher-order singular
  perturbations of non-convex functionals: free-discontinuity problems, 2024.
\newblock ArXiv:2402.10656.

\end{thebibliography}


\end{document}